\documentclass[11pt]{amsart}
\usepackage{enumerate}
\usepackage{hyperref}
\usepackage{amsmath,amssymb,amsthm}
\usepackage{tikz}
\usepackage[T1]{fontenc}
\usepackage[utf8]{inputenc}
\usetikzlibrary{decorations.pathmorphing}
\usepackage{algorithm}
\usepackage{algpseudocode}
\usepackage[letterpaper, margin=1in]{geometry}
\usepackage{todonotes}

\hypersetup{
    pdftitle = {Erd\H{o}s-P\'osa property of long holes},
   pdfauthor=  {Tony Huynh and O-joung Kwon}
   }

\newcommand\abs[1]{\lvert #1\rvert}
\newtheorem{theorem}{Theorem}[section]
\newtheorem{lemma}[theorem]{Lemma}

\newtheorem{corollary}[theorem]{Corollary}
\newtheorem{proposition}[theorem]{Proposition}
\newtheorem{conjecture}[theorem]{Conjecture}

\newtheorem*{claim*}{Claim}
\newtheorem{question}{Question}

\theoremstyle{definition}

\newcommand\dist{\operatorname{dist}}
\newcommand{\bd}{\mathsf{bd}}
\newcommand{\interior}{\mathsf{int}}
\newcommand{\init}{\mathsf{init}}
\newcommand{\EP}{Erd\H{o}s-P\'osa}

\newcommand{\buff}{31}

\newcommand{\s}{\mathcal{S}}
\newcommand{\h}{\mathcal{H}}
\newcommand{\N}{\mathbb{N}}
\newcommand{\R}{\mathbb{R}}
\newcommand{\cF}{\mathcal{F}}

\begin{document}
\title[On the Erd\H{o}s-P\'osa property for long holes in $C_4$-free graphs]{On the Erd\H{o}s-P\'osa property for long holes in $C_4$-free graphs}
\author{Tony Huynh}
\author{O-joung Kwon}
\address[Tony Huynh]{School of Mathematics, Monash University, Melbourne, Australia}
\email{tony.bourbaki@gmail.com}
\address[O-joung Kwon]{Department of Mathematics, Incheon National University, Incheon, South Korea \\ Discrete Mathematics Group, Institute~for~Basic~Science~(IBS), Daejeon,~South~Korea}
\email{ojoungkwon@gmail.com}
\thanks{The first author was supported by ERC grant \emph{FOREFRONT} (grant agreement no. 615640) funded by the European Research Council under the EU's 7th Framework Programme (FP7/2007-2013), and currently by the Australian Research Council.  The second author is supported by the National Research Foundation of Korea (NRF) grant funded by the Ministry of Education (No. NRF-2018R1D1A1B07050294) and by the Institute for Basic Science (IBS-R029-C1). }
\date{\today}
\begin{abstract}
We prove that there exists a function $f(k)=\mathcal{O}(k^2 \log k)$ such that for every $C_4$-free graph $G$ and every $k \in \N$, $G$ either contains $k$ vertex-disjoint holes of length at least $6$, or a set $X$ of at most $f(k)$ vertices such that $G-X$ has no hole of length at least $6$.  This answers a question of Kim and Kwon [Erd{\H{o}}s-P{\'o}sa property of chordless cycles and its applications. JCTB 2020].

\end{abstract}
\maketitle

\section{Introduction}
A classic theorem of Erd\H{o}s and P\'osa~\cite{ErdosP1965} asserts that there exists a function $f: \N \to \R$ such that for every graph $G$ and every $k \in\N$, $G$ either contains $k$ vertex-disjoint cycles, or a set $X$ of at most $f(k)$ vertices such that $G-X$ has no cycle. 
The \EP{} theorem has since been extensively generalized, where the objects to be packed are not necessarily cycles, the containment relation is not necessarily the subgraph relation, and the host class is not necessarily the class of all graphs.  These results are too numerous to cite properly, but we refer the interested reader to a recent survey of Raymond and Thilikos~\cite{Raymond2017}.

In this paper, we are interested in the induced subgraph relation.  A graph is \emph{$H$-free} if it has no induced subgraph isomorphic to $H$.   A family $\cF$ of graphs has the \emph{(induced) \EP{} property} if there exists a function $f:\mathbb{N}\to \mathbb{R}$ such that for every graph $G$ and every $k\in \mathbb{N}$, $G$ either contains $k$ vertex-disjoint (induced) subgraphs each isomorphic to a graph in $\cF$, or a set $X$ of at most $f(k)$ vertices such that $G-X$ has no (induced) subgraph  isomorphic to a graph in $\cF$. 

For each $\ell \geq 3$, we let $C_\ell$ denote the cycle of length $\ell$.  A \emph{hole} in a graph is an induced $C_\ell$ with $\ell \geq 4$.  A graph is \emph{chordal} if it has no hole.  The class of chordal graphs is a widely studied class of graphs, in part because many NP-complete problems can be efficiently solved when restricted to chordal graphs.  Given a graph $G$ and $k \in \N$ as input, the {\sc Chordal Deletion} problem asks whether there exists a set $X$ of at most $k$ vertices in $G$ such that $G-X$ is chordal.  Motivated by the {\sc Chordal Deletion} problem, Jansen and Pilipczuk~\cite{JansenP2017} asked whether cycles of length at least $4$ have the induced \EP{} property.  
This was recently answered in the affirmative by Kim and Kwon~\cite{KK20}.  

\begin{theorem}[Kim and Kwon~\cite{KK20}] \label{kimkwon}
There exists a function $g(k)=\mathcal{O}(k^2 \log k)$ such that for every graph $G$ and every $k \in \N$, $G$ either contains $k$ vertex-disjoint holes, or a set $X$ of at most $g(k)$ vertices such that $G-X$ has no hole.
\end{theorem}


Whenever the \EP{} property holds for a family of objects, it is natural to ask if an analogous ``long'' version also holds.  That is, does the \EP{} property still hold if we insist that the objects are of length at least $\ell$, for some constant $\ell$?  For example, by Menger's theorem~\cite{Menger27}, the \EP{} property holds for ($S$-$T$)-paths (with $f(k)=k$) and Montejano and Neumann-Lara~\cite{MN84} showed that it also holds for ($S$-$T$)-paths of length at least $\ell$.  The original \EP{} theorem was extended to long cycles by Birmelé, Bondy, and Reed~\cite{BBR07} (see~\cite{FH14, MNSW17} for improved bounds).  Kakimura, Kawarabayashi, and Marx~\cite{KKM11} showed that $S$-cycles (cycles with at least one vertex in a prescribed vertex set $S$) have the \EP{} property, and this was extended to long $S$-cycles by Bruhn, Joos, and Schaudt~\cite{BJS18}. Huynh, Joos, and Wollan~\cite{HJW19} showed that $(S,T)$-cycles (cycles that use at least one vertex from each of two prescribed vertex sets $S$ and $T$) and long $(S,T)$-cycles both have the \EP{} property.  

These results might lead one to conjecture that whenever the \EP{} property holds, the ``long'' \EP{} property also holds.  However, this turns out to be false for holes, in the following strong sense.  Let $\alpha \in \N$ and $\mathcal F$ be a family of graphs.  We say that $\mathcal F$ has the \emph{$\frac{1}{\alpha}$-integral (induced) \EP{} property} if there exists a function $f: \N \to \R$ so that for every graph $G$ and every $k \in \N$, $G$ either contains $k$ vertex-disjoint (induced) subgraphs, each isomorphic to a graph in $\mathcal F$ and such that every vertex of $G$ is contained in at most $\alpha$ of these (induced) subgraphs, or a set $X$ of at most $f(k)$ vertices such that $G-X$ has no (induced) subgraph in $\mathcal F$.   Kim and Kwon~\cite{KK20} proved that for all $\alpha \in \N$ and $\ell \geq 5$, cycles of length at least $\ell$ do not have the $\frac{1}{\alpha}$-integral induced \EP{} property.  


\begin{theorem}[Kim and Kwon~\cite{KK20}] \label{counterexamples}
Let $\alpha \in \N$ and $\ell \geq 5$. 
There is no function 
 $f: \N \to \R$ such that for every graph $G$ and every $k \in \N$, $G$ either contains $k$ vertex-disjoint holes of length at least $\ell$ such that every vertex of $G$ is contained in at most $\alpha$ of these holes, or a set $X$ of at most $f(k)$ vertices such that $G-X$ has no hole of length at least $\ell$.
\end{theorem}

Theorem~\ref{counterexamples} is rather surprising, since there are many objects that do not have the \EP{} property but do have the $\frac{1}{2}$-integral \EP{} property.  For example, odd cycles do not have the \EP{} property~\cite{DN88} but they do have the $\frac{1}{2}$-integral \EP{} property~\cite{Reed99}.  An important observation is that the examples from Theorem~\ref{counterexamples} have large complete bipartite graphs as induced subgraphs.  Therefore, Kim and Kwon asked whether the induced \EP{} property holds for long cycles, when we restrict to $C_4$-free graphs.  

\begin{question}[Kim and Kwon~\cite{KK20}] \label{ques:EPlong}
For fixed $\ell \geq 6$,  does there exist a function $f: \N \to \R$ such that for every $C_4$-free graph $G$ and every $k \in \N$, $G$ either contains $k$ vertex-disjoint holes of length at least $\ell$, or a set $X$ of at most $f(k)$ vertices such that $G-X$ has no hole of length at least $\ell$?
\end{question}

The main result of this paper is that Question~\ref{ques:EPlong} is true for $\ell=6$.

\begin{theorem} \label{thm:main}
There exists a function $f(k)=\mathcal{O}(k^2 \log k)$ such that for every $C_4$-free graph $G$ and every $k \in \N$, $G$ either contains $k$ vertex-disjoint holes of length at least $6$, or a set $X$ of at most $f(k)$ vertices such that $G-X$ has no hole of length at least $6$.  Moreover, there is a polynomial-time algorithm that, given a $C_4$-free graph $G$ and a positive integer $k$, either finds $k$ vertex-disjoint holes of length at least $6$ in $G$, or a set $X$ of at most $\mathcal{O}(k^2 \log k)$ vertices such that $G-X$ has no hole of length at least $6$.
\end{theorem}

Note that Theorem~\ref{thm:main} implies Theorem~\ref{kimkwon} as follows.  We claim that we may take $g(k)=5(k-1)+f(k)$, where $f(k)$ is the function from Theorem~\ref{thm:main}.  Let $G$ be a graph and $k \in \N$.  Let $\mathcal U$ be a collection of vertex-disjoint holes of $G$, where each $H \in \mathcal U$ has length $4$ or $5$, and $| \mathcal U|$ is maximum with this property.  If $|\mathcal U| \geq k$, then we are done.  Otherwise, let $X$ be the set of vertices in $G$ covered by the holes in $\mathcal U$.  Since $G - X$ is $C_4$-free, by Theorem~\ref{thm:main}, $G-X$ either contains $k$ vertex-disjoint holes of length at least $6$, or a set $X'$ of at most $f(k)$ vertices such that $(G-X)-X'$ has no hole of length at least $6$. If the former holds, we are done.  If the latter holds, then $G-(X \cup X')$ has no holes and $|X \cup X'| \leq 5(k-1)+f(k)$. 

For brevity, we call a hole of length at least $6$ a \emph{long hole}.
Our proof of Theorem~\ref{thm:main} follows the same general strategy as~\cite{KK20}, but with several simplifications. Indeed, even though we need to deal with several additional technical difficulties in finding \emph{long} holes, our proof is still shorter than the proof for holes.  If one only cares about Theorem~\ref{kimkwon}, it is easy to rewrite our proof to obtain an extremely succinct proof for holes.


We begin by reducing Theorem~\ref{thm:main} to the case where there is a given shortest long hole $C$ such that $G-V(C)$ has no long hole and $C$ has $c'k\log k$ vertices for some constant $c'$.
We show that every vertex in $G-V(C)$ either has at most three neighbors that are consecutive on $C$, or it `almost' dominates $C$. 
We let $D$ be the set of vertices that almost dominate $C$. These vertices play a similar role as the \emph{$C$-dominating vertices} in \cite{KK20}. As in~\cite{KK20}, we can show that $D$ is a clique.

In~\cite{KK20}, holes are divided into four different types, depending on whether they intersect $D$ or not, and whether they are contained in the closed neighborhood of $C$ or not. For each of these four types, Kim and Kwon~\cite{KK20} find either $k$ vertex-disjoint holes of that type, or a small vertex set hitting all holes of that type. In our work, we unify and simplify many of the steps in their proof. In particular, we do not separately consider long holes all of whose vertices are close to $C$. We simply distinguish long holes depending on whether they intersect $D$ or not.  For each of these \emph{two} types, we find either $k$ vertex-disjoint long holes of that type, or a small vertex set hitting all holes of that type.

\subsection*{Paper Outline} In Section~\ref{sec:prelim}, we introduce some basic definitions. We reduce Theorem~\ref{thm:main} to the restricted setting discussed above (Theorem~\ref{thm:main2}) in Section~\ref{sec:overview}, and show that our main result follows from Theorem~\ref{thm:main2}.  The rest of the paper is dedicated to proving Theorem~\ref{thm:main2}.  In Section~\ref{sec:structure}, we introduce almost $C$-dominating vertices, and prove some structural lemmas relative to the cycle $C$. In Section~\ref{sec:greedy}, we find a greedy packing of long holes contained in the union of $D$ and the third closed neighborhood of $C$ in $G-D$. We then find a hitting set of size $\mathcal{O}(k\log k)$ for such holes if the greedy packing does not find $k$ vertex-disjoint long holes. In Section~\ref{sec:eardecomposition}, we find a `sparse' ear decomposition $\h$ of $G$. We argue that the number of branching points (vertices of degree $3$) of $\h$ is  $\mathcal{O}(k\log k)$, otherwise we can find $k$ vertex-disjoint long holes. We remove all branching points of $\h$ together with some vertices of $C$ close to the branching points. We complete the proof of Theorem~\ref{thm:main2} in Section~\ref{sec:avoiding}. Long holes that intersect $D$ are handled in Subsection~\ref{subsec:avoding}, and long holes that avoid $D$ are handled in Subsection~\ref{subsec:traversing}.  We end the paper with some open problems in Section~\ref{sec:openproblems}.

\section{Basic Definitions} \label{sec:prelim}

All graphs in this paper are finite, undirected, and have no loops and parallel edges. 
Let $G$ and $H$ be graphs.  We denote by $V(G)$ and $E(G)$ the vertex set and the edge set of $G$, respectively.  We let $G \cup H$ and $G \cap H$ be the graphs with vertex set $V(G) \cup V(H)$ (respectively, $V(G) \cap V(H)$) and edge set $E(G) \cup E(H)$ (respectively, $E(G) \cap E(H)$).  If $H$ is a subgraph of $G$ and $S \subseteq V(G)$, we let $H-S$ be the graph
obtained from $H$ by removing $S \cap V(H)$ and all edges with at least one end in $S$.  We abbreviate $H-\{v\}$ as $H-v$, and we let $H[S]:=H - (V(G) \setminus S)$.  A subgraph of $G$ is \emph{induced} if it is equal to $G[S]$ for some $S\subseteq V(G)$.  We say that a set $S$ of vertices in $G$ is a \emph{clique} if every pair of vertices in $S$ is adjacent.  
 
We say that $u$ is a \emph{neighbor} of $v$ if $uv \in E(G)$.  
The \emph{open neighborhood} of $A \subseteq V(G)$ is the set of vertices in $V(G)\setminus A$ having a neighbor in $A$, and is denoted $N_G(A)$. 
The set $N_G[A]:=N_G(A)\cup A$ is the \emph{closed neighborhood} of $A$.  For $v \in V(G)$, we let $N_G(v):=N_G(\{v\})$ and $N_G[v]:=N_G[\{v\}]$. For a subgraph $H$ of $G$, we let $N_G(H):=N_G(V(H))$ and $N_G[H]:=N_G[V(H)]$. The \emph{degree} of $v$, denoted $\deg_G(v)$, is $|N_G(v)|$.

We let $\mathbb{N}$ be the set of positive integers, and for all $n\in \mathbb{N}$, let $[n]:=\{1, \dots, n\}$.   
The \emph{length} of a path in $G$ is its number of edges.
For two vertices $x$ and $y$ in $G$, an \emph{$(x,y)$-path} is a path in $G$ whose ends are $x$ and $y$. The \emph{distance} between $x$ and $y$ is the length of a shortest $(x,y)$-path, and is denoted $\dist_G(x,y)$. If there is no $(x,y)$-path in $G$, then $\dist_G(x,y):=\infty$. The distance between two vertex sets $X,Y\subseteq V(G)$, written as $\dist_G(X,Y)$, is the minimum $\dist_G(x,y)$ over all $x\in X$ and $y\in Y$. For $A \subseteq V(G)$ and $r \in \N$, we let $N_G^r[A]$ denote the set of all vertices $w$ such that $\dist_G(\{w\}, A)\le r$. We abbreviate $\dist_G(\{x\},Y)$ as $\dist_G(x,Y)$, and $N_G^r[\{x\}]$ as $N_G^r[x]$. For a subgraph $H$ of $G$, let $N_G^r[H]:=N_G^r[V(H)]$.

For a path $P$ and two vertices $x$ and $y$ in $P$, 
we denote by $xPy$ the subpath of $P$ from $x$ to $y$.
For two walks $P$ and $Q$ such that the last vertex of $P$ is the first vertex of $Q$, we let 
$PQ$ be the concatenation of $P$ and $Q$.  In the special case that $P$ and $Q$ are paths with $x,y \in V(P)$ and $y,z \in V(Q)$, we abbreviate $(xPy)(yQz)$ as $xPyQz$.
For $A \subseteq V(G)$, an \emph{$A$-path} is a path in $G$ such that 
its ends are in $A$, and all other vertices are contained in $V(G)\setminus A$. For a subgraph $H$ of $G$, a $V(H)$-path is simply called an \emph{$H$-path}.

We will need the following result of Simonovitz~\cite{Simonovits1967}, which is useful for finding many vertex-disjoint cycles in a graph of maximum degree $3$. We define $s_k$ for $k\in \mathbb{N}$ as
\begin{align*}
s_k=
\begin{cases}
4k(\log k + \log \log k +4) \quad &\text{if } k\geq 2\\
2 &\text{if } k=1.
\end{cases}
\end{align*}
\begin{theorem}[Simonovitz~\cite{Simonovits1967}]\label{thm:simonovitz}
Let $G$ be a graph  all of whose vertices have degree $3$ and let $k \in \N$. If $\abs{V(G)}\geq s_k$, then $G$ contains $k$ vertex-disjoint cycles. Moreover, these $k$ cycles can be found in polynomial time.
\end{theorem}

\section{An Equivalent Formulation}\label{sec:overview}
Recall that a hole of $G$ is \emph{long} if it has length at least $6$. 
In this section, we show that to prove Theorem~\ref{thm:main}, it suffices to prove the following theorem.


\begin{theorem}\label{thm:main2}
Let $G$ be a $C_4$-free graph, $k \in \N$, and $C$ be a cycle in $G$ such that $C$ is a shortest long hole of $G$, $|V(C)| > \mu_k:=88575k+24003s_k$, and $G-V(C)$ does not contain a long hole.  Then there exists a polynomial-time algorithm that finds $k$ vertex-disjoint long holes of $G$ or a set $X_{\ref{thm:main2}}$ of at most $\mu_k$ vertices in $G$ such that $G-X_{\ref{thm:main2}}$ has no long hole.  
\end{theorem}

We first need an algorithm to detect a shortest long hole, if one exists.


\begin{lemma} \label{lem:detecting}
Let $G$ be a graph, $Q$ be a $(u, v)$-induced path in $G$ of length at least $4$ for some $u,v\in V(G)$, and $P$ be a shortest $(u, v)$-path in $G-(N_G[V(Q) \setminus \{u, v\}] \setminus \{u, v\})$.  Then $P \cup Q$ is a long hole.  Moreover, among all long holes of $G$ containing $Q$, $P \cup Q$ is shortest.
\end{lemma}
\begin{proof}
Since $P$ is a shortest path, it is induced.  Moreover, since $u$ and $v$ are not adjacent, $P$ has length at least $2$, and hence $P \cup Q$ has length at least $6$.  Since there is no edge from $V(Q) \setminus \{u, v\}$ to $V(P) \setminus \{u,v\}$ in $G$, $P \cup Q$ is a long hole of $G$ containing $Q$. To see that $P \cup Q$ is a shortest such hole, observe that if $H$ is any hole of $G$ containing $Q$, then $H - (V(Q) \setminus \{u, v\})$ is a $(u, v)$-path in $G-(N_G[V(Q) \setminus \{u, v\}] \setminus \{u, v\})$.
\end{proof}

\begin{lemma} \label{lem:shortest}
Given a graph $G$, a shortest long hole of $G$ (if one exists) can be found in time $\mathcal{O}(\abs{V(G)}^7)$.
\end{lemma}
\begin{proof}
We guess a $5$-tuple of vertices $\mathcal{X}:=(x_1, x_2, x_3, x_4, x_5)$ that form an induced path $x_1x_2x_3x_4x_5$, 
and find a shortest $(x_5,x_1)$-path $P_{\mathcal{X}}$ in $G-(N_G[\{x_2, x_3, x_4\}] \setminus \{x_1, x_5\})$ (if one exists).
By Lemma~\ref{lem:detecting}, this process finds the shortest long hole of $G$ containing the path $x_1x_2x_3x_4x_5$ (if such a hole exists).  
Therefore, if $P_{\mathcal X}$ does not exist for all $\mathcal X$, then $G$ does not contain a long hole.  Otherwise, if we choose $\mathcal{X}=(x_1, x_2, x_3, x_4, x_5)$, such that $|P_\mathcal{X}|$ is minimum, then $x_5 P_{\mathcal X}x_1x_2x_3x_4x_5$ is a shortest long hole of $G$.  Since shortest paths can be computed in time $\mathcal{O}(|V(G)|^2)$, the above algorithm runs in time $\mathcal{O}(\abs{V(G)}^7)$.
\end{proof}

We are now ready to prove Theorem~\ref{thm:main} assuming that Theorem~\ref{thm:main2} holds. This argument has been repeatedly used in various Erd\H{o}s-P\'osa type results. 

\begin{proof}[Proof of Theorem~\ref{thm:main}]
    By repeatedly applying Lemma~\ref{lem:shortest}, we can construct a sequence of graphs $G_1,\ldots, G_{\ell+1}$ and a sequence of cycles $C_1,\ldots , C_{\ell}$ in polynomial time 
such that 
\begin{itemize}
\item $G_1=G$, 
\item for each $i\in [\ell]$, $C_i$ is a shortest long hole of $G_i$,  
\item for each $i\in [\ell]$, $G_{i+1}=G_i-V(C_i)$, and
\item $G_{\ell+1}$ has no long holes.
\end{itemize}
If $\ell \geq k$, then we found a  packing of $k$ long holes. Thus, we may assume that $\ell \leq k-1$. The following claim  will complete the proof.
\begin{claim*}
For each $j \in [\ell+1]$, we can find in polynomial time either $k$ vertex-disjoint long holes of $G$, or 
a vertex set $X_{j}$ of $G_{j}$ of size at most $(\ell+1-j)\mu_k$ so that $G_j-X_j$ has no long holes. 
\end{claim*}

\begin{proof}[Subproof]
We proceed by reverse induction on $j$.  The base case of $j=\ell+1$ holds with $X_{\ell+1}=\emptyset$ since $G_{\ell+1}$ has no long holes.  Fix $j\leq \ell$.  By induction, there is a subset of vertices $X_{j+1}$ of $G_{j+1}$ of size at most $(\ell-j)\mu_k$ hitting all long holes of $G_{j+1}$.  Since $G_{j+1}=G_j-V(C_j)$, $C_{j}$ is a shortest long hole of $G_{j}-X_{j+1}$, and $\left( G_{j}-X_{j+1} \right)-V(C_{j})$ has no long holes.

If $C_{j}$ has length at most $\mu_k$, then we set $X_{j}:=X_{j+1}\cup V(C_{j})$. Clearly, $\abs{X_j}\leq (\ell+1-j)\mu_k$.
Otherwise, by applying Theorem~\ref{thm:main2} to $(G_{j}-X_{j+1}, k, C_{j})$, 
we can find in polynomial time either $k$ vertex-disjoint long holes of $G_{j}-X_{j+1}$, or a vertex set $X$ of size at most $\mu_k$ of $G_{j}-X_{j+1}$ hitting all long holes. 
In the former case, we output $k$ vertex-disjoint holes, and we are done. In the latter case, we set $X_{j}:=X_{j+1}\cup X$. Observe that $G_{j}- X_j$ has no long holes and $\abs{X_{j}}\le (\ell+1-j)\mu_k$ as claimed. 
\end{proof}
In particular, the set $X_1$ hits every long hole of $G_1=G$ and 
$$\abs{X_1}\le \ell\mu_k\le (k-1)\mu_k=\mathcal{O}(k^2\log k). \qedhere$$ 
\end{proof}

\section{Structure Relative to the Cycle $C$}\label{sec:structure}
For the remainder of the paper, we fix a $C_4$-free graph $G$, $k \in \N$, and a cycle $C$ in $G$ such that $C$ is the shortest long hole of $G$,  $|V(C)| > \mu_k$, and $G-V(C)$ does not contain a long hole.  Our goal is to prove Theorem~\ref{thm:main2}.  We begin by analyzing the structure of $G$ relative to $C$.

\subsection{Almost $C$-dominating vertices}

We say that a vertex $u \notin V(C)$ is \emph{almost $C$-dominating} if the neighbors of $u$ on $C$ can be enumerated $u_1, \dots, u_\ell$ (with $u_{\ell+1}=u_1$) such that $\ell \geq 2$ and $\dist_C(u_i, u_{i+1}) \in \{1,3\}$ for all $i \in [\ell]$.  
This definition is motivated by the following lemma.

\begin{lemma}\label{lem:consecutive}
For every vertex $v \in V(G)\setminus V(C)$, either $v$ is almost $C$-dominating or $\abs{N_G(v)\cap V(C)} \leq 3$ and the vertices in $N_G(v)\cap V(C)$ are consecutive along $C$.
\end{lemma}

\begin{proof}
Towards a contradiction, let $v \in V(G)\setminus V(C)$ be a vertex for which neither outcome occurs. Since $v$ is not almost $C$-dominating and $G$ is $C_4$-free, there is a $(v_1, v_2)$-subpath $P$ of $C$ such that $v_1, v_2 \in N_G(v)\cap V(C)$, no internal vertex of $P$ is in $N_G(v)\cap V(C)$, and $|V(P)| \geq 5$.  If $|V(C)|-|V(P)|=0$, then the second outcome occurs, which contradicts that $v$ is a counterexample.  Suppose $V(C) \setminus V(P):=\{x\}$.  If $x \in N_G(v)$, then the second outcome occurs again, which is a contradiction.  If $x \notin N_G(v)$, then $vv_1xv_2v$ is an induced $C_4$ of $G$, which is also a contradiction. Thus, $|V(C)|-|V(P)| \geq 2$.  But now, $v_1Pv_2vv_1$ is a long hole of length strictly less than $|V(C)|$, contradicting the assumption that $C$ is a shortest long hole in $G$.
\end{proof}

We next show that the set of almost $C$-dominating vertices is a clique in $G$.

\begin{lemma}\label{lem:cdominating}
The set of almost $C$-dominating vertices is a clique.
\end{lemma}

\begin{proof}
Let $a$ and $b$ be almost $C$-dominating vertices and towards a contradiction suppose that $a$ and $b$ are not adjacent.
Let $P_1, P_2, P_3$ be subpaths of $C$ such that each $P_i$ has length $2$ and $\dist_C(P_i, P_j)\ge 2$ for distinct $i,j\in [3]$.
We can choose such paths because $\mu_k\ge 11$ and $\abs{V(C)}>\mu_k$.
By the definition of an almost $C$-dominating vertex, 
$a$ has a neighbor in each of $P_1, P_2, P_3$, and $b$ has a neighbor in each of $P_1, P_2, P_3$.

If $a$ and $b$ have common neighbors in two of the $P_i$, then $G$ has a hole of length $4$, a contradiction.
Therefore, by symmetry, we may assume that  $a$ and $b$ have no common neighbors in $V(P_1 \cup P_2)$.  For $i\in [2]$, let $P_i^{ab}$ be a shortest $(a,b)$-path in $G[\{a,b\}\cup V(P_i)]$.  Then $C':=P_1^{ab} \cup P_2^{ab}$ is a hole of length at least $6$ and at most $8 < \mu_k < \abs{V(C)}$, which contradicts the assumption that $C$ is a shortest long hole.
\end{proof}

For the remainder of the paper, we let $D$ be the set of almost $C$-dominating vertices of $G$.  Since every hole of $G$ contains at most two vertices from each clique of $G$, we have the following corollary.

\begin{corollary}\label{cor:clique}
Every hole of $G$ contains at most two vertices of $D$.
\end{corollary}

\subsection{Cycle-distance lemma}\label{subsec:distance}
	We now give a lower bound on the length of a path in $G-D$ such that its ends are on $C$ and are far apart in $C$. 
Lemmas~\ref{lem:distancebasic1} and \ref{lem:distancebasic2} for $C$-paths are essentially shown in \cite[Section 4.1]{KK20}, with a similar assumption. 
 We extend the argument for $C$-paths to all paths in $G-D$ in Lemma~\ref{lem:distancebasic3}.  This lemma will be used repeatedly throughout the paper.  


\begin{lemma}\label{lem:distancebasic1}
If $x$ and $y$ are vertices of $C$ such that $\dist_C(x,y)\ge 4$, then  
every $(x,y)$-path in $G-D$ has length at least $4$.
\end{lemma}
\begin{proof}
Towards a contradiction, suppose there is an $(x,y)$-path $R$ in $G-D$ of length at most $3$.
Let $P$ and $Q$ be the two $(x,y)$-paths in $C$ such that $\abs{V(P)}\le \abs{V(Q)}$.
Since $\mu_k\ge 16$, $Q$ has length at least $8$. 

Since $C$ is a hole and $\dist_C(x,y)\ge 4$, $R$ has an internal vertex contained in $V(G)\setminus V(C)$.
By Lemma~\ref{lem:consecutive}, every vertex of $V(R)\setminus V(C)$ has at most $3$ neighbors on $C$ and these neighbors are consecutive along $C$. So, $R$ has length $3$. If one of the internal vertices of $R$ is in $C$, then the other internal vertex has two neighbors in $C$ that have distance at least $3$ in $C$, contradicting Lemma~\ref{lem:consecutive}. Therefore, every internal vertex of $R$ is not contained in $C$.

Now, $G[V(Q)\cup V(R)]$ contains a hole $C'$ of length at least $8-4+3=7$.
Since $P$ contains at least three internal vertices while $R$ contains at most two internal vertices, 
$C'$ is shorter than $C$, a contradiction.
\end{proof}

\begin{lemma}\label{lem:distancebasic2}
Let $m \in \N$, and $P$ be a $C$-path in $G-D$ with ends $x$ and $y$.
If $\dist_C(x,y)\geq  4m$, then $P$ has length at least $m+3$.
\end{lemma}
\begin{proof}
We proceed by induction on $m$.
The base case $m=1$ follows from Lemma~\ref{lem:distancebasic1}. 
Fix $m\ge 2$ and towards a contradiction let $P:=p_1p_2 \cdots p_n$ be a $C$-path of length at most $m+2$ in $G-D$ such that  $\dist_C(p_1,p_n)\geq 4m$.  We may clearly assume that $P$ is a shortest such path.  In particular, $P$ is an induced path.  
Note that $n-1\le m+2$ and thus, $m\ge n-3$.
Since $\dist_C(x,y)\ge 4m\ge 4$, Lemma~\ref{lem:distancebasic1} implies that $P$ has length at least $4$.
We distinguish cases depending on whether $\{p_j : j\in [n-2]\setminus [2]\}$ contains a vertex in $N_G(C)$ or not.

\medskip
\noindent {\bf Case 1.} $\{p_j : j\in [n-2]\setminus [2]\}$ contains a vertex in $N_G(C)$. \\
We choose an integer $i\in [n-2]\setminus [2]$ such that $p_i\in N_G(C)$, and
choose a neighbor $z$ of $p_i$ in $C$.
Since $p_1p_2 \cdots p_iz$ is a $C$-path of length $i<m$, by the induction hypothesis, $\dist_C(p_1, z)< 4(i-2)$.
Since $zp_ip_{i+1} \cdots p_{n-1}p_n$ is a $C$-path of length $n-i+1\le m+4-i\le m+1$, 
by the induction hypothesis, we have $\dist_C(z, p_n)<4(n-i+1-2)=4(n-i-1)$.
Therefore, we have 
\[\dist_C(p_1, p_n)\le \dist_C(p_1, z)+\dist_C(z, p_n)<4(n-3)\le 4m,\]
a contradiction.

\medskip
\noindent {\bf Case 2.} $\{p_j : j\in [n-2]\setminus [2]\}$ contains no vertices in $N_G(C)$.
\\
Let $Q$ be a shortest path from $N_G(p_2)\cap V(C)$ to $N_G(p_{n-1})\cap V(C)$ in $C$, and let $q, q'$ be its ends.
Let $C':= p_2Pp_{n-1}q'Qqp_2$.  By assumption, there are no edges between $\{p_j : j\in [n-2]\setminus [2]\}$ and $V(Q)$.  By the minimality of $|V(Q)|$ there are no edges between $\{p_2, p_{n-1}\}$ and $V(Q)$.  Finally, since $P$ and $Q$ are induced paths, it follows that $C'$ is a hole.  Moreover, because neither $p_2$ nor $p_{n-1}$ are in $D$, Lemma~\ref{lem:consecutive} implies $\dist_C(q,q') \geq \dist_C(p_1, p_n)-4 \geq 4m -4 \geq 4$.  Therefore, $|V(C')| \geq 6$.  

Observe that $\dist_C(p_1,p_n)\le \frac{\abs{V(C)}}{2}$ and 
$m\le \frac{\dist_C(p_1,p_n)}{4}\le \frac{\abs{V(C)}}{8}$. Thus, $P$ has length at most $m+2\le \frac{\abs{V(C)}}{8}+2$. 
Therefore, 
\[ \abs{V(C')}=\abs{V(Q)}+\abs{(V(P)\setminus \{p_1, p_n\})}\le \frac{\abs{V(C)}}{2}+ 1 +  \frac{\abs{V(C)}}{8} +1 < \abs{V(C)}.\]
This contradicts that $C$ is a shortest long hole of $G$.

\medskip
This concludes the proof.
\end{proof}

  \begin{lemma}\label{lem:distancebasic3}
Let $m \in \N$, and $P$ be a path in $G-D$ with ends $x$ and $y$ in $C$.
If $\dist_C(x,y)\geq  4m$, then $P$ has length at least $m+3$.
\end{lemma}
\begin{proof}
We proceed by induction on $m$.  The base case $m=1$ follows from Lemma~\ref{lem:distancebasic1}. Fix $m \geq 2$ and towards a contradiction let $P=p_1p_2 \cdots p_n$ with $p_1=x, p_n=y$, $n \leq m+3$, and $\dist_C(x,y)\geq  4m$.  We may assume that $P$ contains a vertex of $C$ as an internal vertex, otherwise $P$ is a $C$-path, and we are done by Lemma~\ref{lem:distancebasic2}.

Let $i\in [n-1]\setminus [1]$ be such that $p_i\in V(C)$.
By the induction hypothesis, $\dist_C(p_1, p_i)< 4(i-3)$, and
 $\dist_C(p_i, p_n)< 4(n-i-2)$.
 Thus, $\dist_C(p_1, p_n)<4(i-3)+4(n-i-2)=4(n-5)<4m$.
\end{proof}

\section{First Greedy Packing of Long Holes}\label{sec:greedy}
In this section we construct a first greedy packing of long holes in $G$.  For a vertex $v \in V(C)$ and a positive integer $r$, let $\s^r_v:= N^r_{G-(D \cup (V(C)\setminus \{v\}))} [v]$. For $X \subseteq V(C)$ and $r \in \N$, let $\s^r_X:= \bigcup_{v \in X} \s^r_v$ and for a subgraph $H$ of $C$, let $\s^r_H:=\s^r_{V(H)}$.  

\medskip

\textbf{Algorithm. (First greedy packing)}

\begin{enumerate}
	\item Choose a vertex $v_{\init} \in V(C)$ to hit $C$. Initialize $C_1=C-v_{\init}$, $D_1=D$, and $M_1=\{v_{\init}\}$. 
    \item Choose a previously unchosen pair $(A,B)$, where $A$ is a subset of size at most $2$ of $D_i$ (possibly, $A=\emptyset$) and $B$ is a set of at most $60$ consecutive vertices of $C$ contained in $C_i$.  If there are no remaining pairs, then proceed to step (6).
    \item Test if $G[A \cup \s^3_B]$ contains a long hole.   
    \item If no, then proceed to step (2).  
    \item If yes, let $H_i$ be a long hole of $G[A \cup \s^3_B]$ and set $D_{i+1}:=D_i \setminus A$ and $C_{i+1}:=C_i - N_C^{15}[B]$ and $M_{i+1}:=M_i\cup A\cup N_C^{15}[B]$, then proceed to step (1).  
    \item Let $\ell$ be the largest index for which $D_{\ell}$ exists. Define $X_{\ref{lem:greedypacking}}:=(M_{\ell}\cap D)\cup N_C^{60}[M_{\ell}\cap V(C)]$.
\end{enumerate}

\begin{lemma} \label{lem:greedypacking}
Let $\ell$ be the largest index for which $D_\ell$ exists in the first greedy packing algorithm.  Then $\{H_i:i\in [\ell-1]\}$ is a set of vertex-disjoint long holes of $G$ and $\abs{X_{\ref{lem:greedypacking}}}\le 212\ell$.
Moreover, for all subpaths $Q$ of $C$ on at most $60$ vertices, no long hole of $G-X_{\ref{lem:greedypacking}}$ is contained in $\s^3_Q\cup D$.
\end{lemma}
\begin{proof}
Let $i,j$ be distinct indices in $[\ell-1]$ and let $(A_i, B_i)$ and $(A_j, B_j)$ be the pairs considered by the algorithm when the holes $H_i$ and $H_j$ were constructed, respectively.
By construction, $A_i\cap A_j=\emptyset$ and $\dist_C(B_i, B_j) \geq 16$. 
If $\s^3_{B_i}$ and $\s^3_{B_j}$ have a common vertex, then there is a $C$-path in $G-D$ of length at most $6$ that starts at $B_1$ and ends at $B_2$.
Therefore, by Lemma~\ref{lem:distancebasic3}, $\dist_C(B_1, B_2)<16$, which is a contradiction.
Thus, $\s^3_{B_1}$ and $\s^3_{B_2}$ are disjoint, and so $H_i$ and $H_j$ are vertex-disjoint.

It is straightforward to check that $\abs{M_{\ell}\cap D}\le 2(\ell-1)$, $\abs{M_{\ell}\cap V(C)}\le 90(\ell-1)+1$, and  $C[M_{\ell}\cap V(C)]$ consists of at most $\ell$ connected components. Therefore, $\abs{X_{\ref{lem:greedypacking}}}\le 212\ell$.

For the second statement, suppose that $G-X_{\ref{lem:greedypacking}}$ has a long hole $H$ contained in $\s^3_Q\cup D$ for some subpath $Q$ of $C$  on at most $60$ vertices.
Because $G-V(C)$ has no long holes, $H$ contains a vertex in $V(C)\setminus X_{\ref{lem:greedypacking}}$.
Since $N^{60}_C[M_{\ell}\cap V(C)] \subseteq X_{\ref{lem:greedypacking}}$, 
\[\dist_C(V(H)\cap V(C), M_{\ell}\cap V(C))\ge 61, \]
and $Q$ is disjoint from $M_{\ell}\cap V(C)$. Therefore, $Q$ is fully contained in $C_{\ell}$. 
By considering the pair $(V(H)\cap D, V(Q) )$ in the algorithm, we should have proceeded one more step, which is a contradiction.
\end{proof}

The next lemma extends the second part of Lemma~\ref{lem:greedypacking}, by showing that for all subpaths $Q$ of $C-X_{\ref{lem:greedypacking}}$, there is no long hole of $G-X_{\ref{lem:greedypacking}}$ contained in $\s^3_Q$. We stress that there could still be a subpath $Q$ of $C$ such that $G-X_{\ref{lem:greedypacking}}$ has a long hole contained $\s^3_Q$. However, these remaining holes will be dealt with in Section~\ref{sec:avoiding}. 

    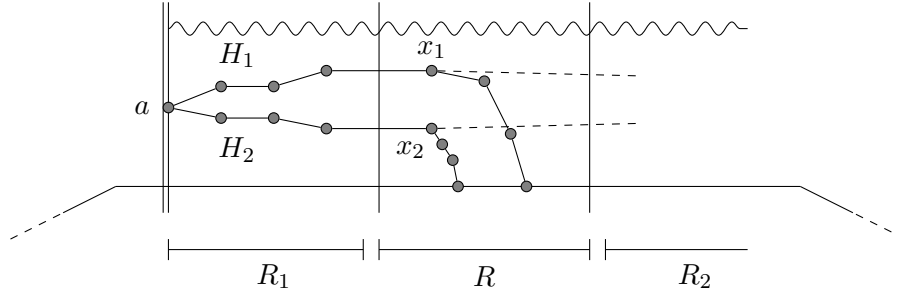
\begin{figure}
  \centering
  \begin{tikzpicture}[scale=0.7]
  \tikzstyle{w}=[circle,draw,fill=black!50,inner sep=0pt,minimum width=4pt]

   \draw (-2,0)--(11,0);
	\draw(-2, 0)--(-3,-0.5);
	\draw(11, 0)--(12,-0.5);
    \draw[dashed](13, -1)--(12,-0.5);
	\draw[dashed](-4,-1)--(-3,-0.5);

     \node at (1, -1.7) {$R_1$};
     \node at (5, -1.7) {$R$};
     \node at (9, -1.7) {$R_2$};
    
     \draw (3, -0.5)--(3, 3.5);
     \draw (7, -0.5)--(7, 3.5);

     \draw (-1, -0.5)--(-1, 3.5);
     \draw (-1.1, -0.5)--(-1.1, 3.5);

   \draw (-1, -1)--(-1, -1.4);
   \draw (2.7, -1)--(2.7, -1.4);
   \draw (3, -1)--(3, -1.4);
   \draw (7, -1)--(7, -1.4);
   \draw (7.3, -1)--(7.3, -1.4);
   \draw (-1, -1.2)--(2.7, -1.2);
 \draw (3, -1.2)--(7, -1.2);
 \draw (7.3, -1.2)--(10, -1.2);

    \draw[decorate, decoration={snake}] (-1, 3)--(10, 3);
 	
    \draw (-1, 1.5) node [w] (a) {};
\draw (0, 1.9) node [w] (v1) {};
\draw (1, 1.9) node [w] (v2) {};
\draw (2, 2.2) node [w] (v3) {};
\draw (4, 2.2) node [w] (v4) {};
\draw (5, 2) node [w] (v5) {};
\draw (5.5, 1) node [w] (v6) {};
\draw (5.8, 0) node [w] (v7) {};

\draw (0, 1.3) node [w] (w1) {};
\draw (1, 1.3) node [w] (w2) {};
\draw (2, 1.1) node [w] (w3) {};
\draw (4, 1.1) node [w] (w4) {};
\draw (4.2, 0.8) node [w] (w5) {};	
\draw (4.4, 0.5) node [w] (w6) {};	
   \draw (4.5, 0) node [w] (w7) {};

     \node at (-1.5, 1.5) {$a$};

     \node at (0.3, 2.5) {$H_1$};
     \node at (0.3, 0.7) {$H_2$};

     \node at (4, 2.6) {$x_1$};
     \node at (3.6, 0.7) {$x_2$};
      
      \draw[dashed] (w4)--(8, 1.2);
      \draw[dashed] (v4)--(8, 2.1);
      
     \draw(a)--(w1)--(w2)--(w3)--(w4)--(w5)--(w6)--(w7);
     \draw(a)--(v1)--(v2)--(v3)--(v4)--(v5)--(v6)--(v7);
     
  \end{tikzpicture}     \caption{The paths $H_1$ and $H_2$ in Lemma~\ref{lem:largesupp}.}\label{fig:greedy1}
\end{figure}

\begin{lemma}\label{lem:largesupp}
For all subpaths $Q$ of $C-X_{\ref{lem:greedypacking}}$, there is no long hole of $G-X_{\ref{lem:greedypacking}}$ contained in $\s^3_Q$.
\end{lemma}
\begin{proof}
For a contradiction, suppose that
$G-X_{\ref{lem:greedypacking}}$ contains a long hole $H$ in $\s^3_Q$ for some subpath $Q$ of $C-X_{\ref{lem:greedypacking}}$.
We assume that $Q$ is a shortest such path.
By Lemma~\ref{lem:greedypacking}, $\abs{V(Q)}\ge 61$.
Let $Q:=q_1q_2 \cdots q_m$. By the minimality of $|V(Q)|$, 
\[(\s^3_{q_1}\setminus \s^3_{Q-q_1})\cap V(H)\neq \emptyset \quad \text{and} \quad (\s^3_{q_m}\setminus \s^3_{Q-q_m})\cap V(H)\neq \emptyset.\]
Let $a\in (\s^3_{q_1}\setminus \s^3_{Q-q_1})\cap V(H)$ and $b\in (\s^3_{q_m}\setminus \s^3_{Q-q_m})\cap V(H)$.  Clearly, we have $a \neq b$.  Let $H_1$ and $H_2$ be the two $(a,b)$-paths in $H$. See Figure~\ref{fig:greedy1} for an illustration.

Let $R:=q_{21}q_{22} \cdots q_{40}$.
Let $R_1$ and $R_2$ be the two components of $Q-V(R)$ containing $q_1$ and $q_m$ respectively.
If $a\in \s^3_R$, then $G-D$ contains a $(q_1, r)$-path of length at most $6$, for some $r \in V(R)$. However, since $\dist_C(q_1, V(R))\ge 16$, this contradicts Lemma~\ref{lem:distancebasic3}. For the same reason, $b\notin \s^3_R$.

We claim that for each $i\in [2]$, $V(H_i)\cap \s^3_R\neq \emptyset$.
Suppose for a contradiction that $V(H_i)\cap \s^3_R= \emptyset$.
Because $\s^3_{R_1}\cap V(H_i)\neq \emptyset$, $\s^3_{R_2}\cap V(H_i)\neq \emptyset$, and $V(H_i)\cap \s^3_R= \emptyset$, 
there must be an edge $uv$ in $H_i$ such that $u\in \s^3_{R_1}\cap V(H_i)$ and $v\in \s^3_{R_2}\cap V(H_i)$.  Thus, $G-D$ contains an $(r_1, r_2)$-path of length at most $7$, for some $r_1 \in V(R_1)$ and $r_2 \in V(R_2)$.  However, this contradicts Lemma~\ref{lem:distancebasic3}, since $\dist_C(R_1, R_2)\ge 20$. 

For each $i\in [2]$, let $x_i$ be the vertex in $\s^3_R\cap V(H_i)$ such that $\dist_{H_i}(a, x_i)$ is minimum.  Let $H':=x_1H_1aH_2x_2$.
If $\dist_{H_i}(a, x_i)\leq 1$, then there is a $(q_1, r)$-path of length at most $7$ in $G-D$, for some $r \in V(R)$.  
Since $\dist_C(q_1, R)\ge 20$, this contradicts Lemma~\ref{lem:distancebasic3}.  Thus, $H'$ has length at least $4$.
Let $P_i$ be a path of length at most $3$ from $x_i$ to $R$ in $\s^3_R$ for each $i\in [2]$.  
Observe that there is no edge between $(P_1-x_1)\cup (P_2-x_2)\cup R$ and $H'-\{x_1, x_2\}$, because of the choice of $x_1$ and $x_2$.
Therefore, if we let $P$ be a shortest $(x_1, x_2)$-path in $G[V(P_1)\cup V(P_2)\cup V(R)]$, then  
$H' \cup P$ is a long hole of $G-X_{\ref{lem:greedypacking}}$. However, $H' \cup P$ is contained in $\s^3_{q_1q_2 \cdots q_{40}}$, contradicting Lemma~\ref{lem:greedypacking}. 
\end{proof}

\section{Sparse Ear Decompositions}\label{sec:eardecomposition}
   A \emph{sparse ear decomposition} in $G$ is an ear decomposition $C \cup P_1 \cup  \dots \cup P_\ell$ of a subgraph $\h$ of $G$ satisfying the following conditions.

\begin{enumerate}[(i)]
    \item $C$ is the cycle of $G$ given by Theorem~\ref{thm:main2}, 
    \item $\h$ has maximum degree $3$, and
    \item for each $i \in [\ell]$, $P_i$ has an end $p_0$ in $C$ such that $P_i:=p_0p_1 \cdots p_m$ with $m \geq 6$, $P_i - \{p_0, p_m\}$ is an induced path in $G$, and 
   no vertex in $\{p_2,p_3,p_4\}$ has a neighbor in $C\cup P_1\cup \cdots \cup P_{i-1}$.
 \end{enumerate}
    A vertex of degree $3$ in a sparse ear decomposition $\h$ is called a \emph{branching point} of $\h$.
 
Let $\h=C \cup P_1 \cup  \dots \cup P_\ell$ be a sparse ear decomposition in $G$.  An $\h$-path $P$ \emph{extends} $\h$ if $C \cup P_1 \cup  \dots \cup P_\ell \cup P$ is a sparse ear decomposition.  We say that $\h$ is \emph{maximal} if there does not exist an $\h$-path extending $\h$.  
We say that $\h$ is \emph{good} if 
for every $i \in [\ell-1]$, 

	\begin{enumerate}
	    \item $P_i$ extends $C\cup P_1\cup \cdots \cup P_{i-1}$, 
		\item subject to (1), $\abs{V(P_i)\cap V(C)}$ is maximum, and 
		\item subject to (1) and (2), $P_i$ is shortest.
		\end{enumerate}

\begin{lemma} \label{lem:sparsealgo}
A maximal good sparse ear decomposition $\h$ in $G$ can be constructed in polynomial time.  
\end{lemma}

\begin{proof}
    Suppose we are given a good sparse ear decomposition $\h=C \cup P_1 \cup  \dots \cup P_\ell$.  We claim that in polynomial time, we can find a path $P$ such that 
    $P$ extends $\h$ and $C \cup P_1 \cup  \dots \cup P_\ell \cup P$ is good, or determine that no such path $P$ exists.  
    To do so, for all $6$-tuples of vertices $(p_0, p_1, p_2, p_3, p_4, p_m)$ of $G$ we define a potential candidate $P(p_0, p_1, p_2,p_3,p_4, p_m)$ for $P$ by first testing if:
	\begin{itemize}
		\item  $p_0\in V(C)$, $p_m\in V(\h)\setminus \{p_0\}$, and $p_0$ and $p_m$ are not branching points of $\h$,
		\item $p_0p_1\in E(G)$, $p_1p_2p_3p_4$ is an induced path in $G-V(\h)$,  and $\{p_2,p_3,p_4\} \cap N_G[\h]=\emptyset$.
	\end{itemize}
	
	If the above conditions fail, $P(p_0, p_1, p_2,p_3,p_4, p_m)$ is undefined.  If the above conditions hold, but there is no path from $p_4$ to $p_m$ in 
	\[
	G':=G-((V(\h)\cup N_G[\{p_1, p_2, p_3\}]) \setminus \{p_m\}),
	\]
	then $P(p_0, p_1, p_2,p_3,p_4, p_m)$ is undefined.  Otherwise, we choose a shortest path $Q$ from $p_4$ to $p_m$ in 
	$G'$, and 
	define $P(p_0, p_1, p_2,p_3,p_4, p_m):=p_0p_1p_2p_3p_4Q$.  If $P(p_0, p_1, p_2,p_3,p_4, p_m)$ is undefined for all ($p_0, p_1, p_2,p_3,p_4, p_m)$, then no such $P$ exists.  Otherwise, we choose $P$ to be the shortest path among 
	all $P(p_0, p_1, p_2,p_3,p_4, p_m)$ which contain the maximum number of vertices of $C$.  

	To complete the proof of the lemma we begin with the good sparse ear decomposition $\h=C$ and repeatedly apply the above claim to obtain a set of paths $P_1, \dots, P_\ell$ such that for all $i \in [\ell]$, 
    \begin{itemize}
        \item $C \cup P_1 \cup \dots \cup P_i$ is a good sparse ear decomposition,
        \item $P_i$ extends  $C \cup P_1 \cup \dots \cup P_{i-1}$, and
        \item  there is no path $P$ such that $P$ extends $C \cup P_1 \cup \dots \cup P_{\ell}$ and $C \cup P_1 \cup \dots \cup P_{\ell} \cup P$ is good.  
    \end{itemize}
   Then $C \cup P_1 \cup \dots \cup P_{\ell}$ is a maximal good sparse ear decomposition of $G$, and can clearly be found in polynomial time.
\end{proof}

	For the remainder of the paper, we fix a maximal good sparse ear decomposition 
	\[
	\h=C\cup P_1\cup \cdots \cup P_\ell.
	\]
    
    \begin{lemma}\label{lem:cycletohole}
    If $F$ is a cycle of $\h$, then $G[V(F)]$ contains a long hole.  
    Moreover, if $\h$ has at least $s_k$ branching points, then 
    we can output $k$ vertex-disjoint long holes of $G$ in polynomial time.  
    \end{lemma}
    \begin{proof}
    For the first part, we may assume that $F$ is a cycle of $\h$ other than $C$, because $C$ is a long hole.
    Since $F$ is a cycle of $\h$ other than $C$,
    $F$ contains an internal vertex of $P_i$ for some $i\in [\ell]$.
    Let $i$ be the maximum integer in $[\ell]$ such that 
    $F$ contains an internal vertex of $P_i$.
    Observe that if $P_j$ has an end which is an internal vertex of $P_i$, then  $j>i$.  Therefore, by the choice of  
     $i$, $V(P_i) \subseteq V(F)$.  
    Let $P_i=p_0p_1 \cdots p_m$ be such that $p_0 \in V(C)$ and  
    no vertex in $\{p_2, p_3, p_4\}$ has a neighbor in $C\cup P_1\cup \cdots \cup P_{i-1}$.
    Then $p_1p_2p_3p_4p_5$ is an induced path 
    and no vertex in $\{p_2, p_3, p_4\}$ has a neighbor in $V(F) \setminus \{p_j: j\in [5]\}$.
    So, by Lemma~\ref{lem:detecting}, $G[V(F)]$ contains a long hole.
    
    For the second part, assume that $\h$ contains $s_{k}$ branching points.
    By Theorem~\ref{thm:simonovitz}, 
    we can find a set $\{F_i: i\in [k]\}$ of $k$ vertex-disjoint cycles of $\h$ in polynomial time.  By the first part, $G[V(F_j)]$ contains a long hole $H_j$ for all $j \in [k]$.  Finally, by Lemma~\ref{lem:shortest}, we can find long holes in $\{H_i: i\in [k]\}$ in polynomial time.   
	\end{proof}
	
	Let $B$ be the set of branching points of $\h$, and define
	\[X_{\ref{lem:sparsealgo}}:=N^{\buff}_C[B\cap V(C)] \cup (B\setminus V(C)).\]

	By Lemma~\ref{lem:cycletohole}, we may assume that $\abs{B}< s_{k}$.    We now prove a special property of $\h$ to be 
     used in Section~\ref{sec:avoiding}.
    
    \begin{lemma} \label{lem:extension}
    For each $j \in [\ell]$, let $\h_j=C \cup P_1 \cup \cdots \cup P_j$.  
    Let $v\in V(C)\setminus X_{\ref{lem:sparsealgo}}$, $u$ be an internal vertex of $P_j$ for some $j\in [\ell]$, and 
    $R=v_0v_1 \cdots v_s$ be an $\h_j$-path with $v_0=v$ and $v_s=u$.
    Then $s\ge 6$.
    \end{lemma}
    
    \begin{proof}
    Let $R=v_0v_1 \cdots v_s$ be a counterexample with $s$ minimum.  
    By the minimality of $s$, 
    if $s\ge 2$, then there is no edge between $\{v_i:i\in \{0\}\cup [s-2]\}$ and the set of internal vertices of $P_j$. 
    Let $P_j=u_0u_1 \cdots u_p$, where $u_0 \in V(C)$ and  
    no vertex in $\{u_2, u_3, u_4\}$ has a neighbor in $C\cup P_1\cup \cdots \cup P_{j-1}$.
    
    Since we added the set $N^{\buff}_C[B\cap V(C)]$ to $X_{\ref{lem:sparsealgo}}$, we have $\dist_C(v_0, B\cap V(C))\ge 32$.
    Let $t$ be the minimum integer such that $v_{s-1}$ is adjacent to $u_{t}$. 

	First assume that $v_{s-1}$ has a neighbor in $\{u_1, u_2, u_3, u_4\}$.
	Then there is a $C$-path from $v_0$ to $u_0$ of length at most $(s-1)+5\le 9$.
	By Lemma~\ref{lem:distancebasic3}, $\dist_C(v_0, u_0)<28$, which contradicts $\dist_C(v_0, B\cap V(C))\ge 32$.
	So, $v_{s-1}$ has no neighbor in $\{u_1, u_2, u_3, u_4\}$.  Therefore
	\[Q:=u_0P_ju_tv_{s-1}Rv_0\] is a path extending $\h_{j-1}$ with $|V(Q) \cap V(C)|=2$.  Since $\h$ is good, this implies $u_p \in V(C)$.
	
	If $p-t>s$, then $Q$ is shorter than $P_j$, a contradiction.
	So, $p-t\le s$, and thus, there is a $C$-path of length at most $s+(p-t)\le 2s\le 10$ from $v_0$ to $u_p$.
	By Lemma~\ref{lem:distancebasic3}, $\dist_C(v_0, u_p)<32$, which contradicts $\dist_C(v_0, B\cap V(C))\ge 32$.
    We conclude that $s\ge 6$, as required.
    \end{proof}
    
       
    Before we proceed further, we deal with cycles with some additional constraint similar to extensions for $\h$, that intersect $C$ on exactly one vertex.
    A cycle $F=v_0v_1v_2 \cdots v_nv_0$ with $n\ge 5$ is an \emph{appendage} of $\h$ if 
    \begin{itemize}
    \item $v_0\in V(C)\setminus B$ and $V(F)\cap V(\h)=\{v_0\}$,
    \item $v_1v_2 \cdots v_n$ is an induced path of $G$, and
    \item no vertex in $\{v_2,v_3,v_4\}$ has a neighbor in $\h$.
    \end{itemize}
    The vertex $v_0$ is called the \emph{tip} of this appendage. 
    Clearly, an appendage contains a long hole, 
     because $v_1v_2v_3v_4v_5$ is an induced path and there are no edges between $\{v_2, v_3, v_4\}$ and $V(F)\setminus \{v_i: i\in [5]\}$.
    We now show that if there are two appendages of $\h$ whose tips have distance at least $20$ in $C$, 
    then they are vertex-disjoint.
    
    \begin{lemma}\label{lem:narrow}
    Let $F_1$ and $F_2$ be appendages of $\h$ with tips $x_1$ and $x_2$, respectively.
    If $\dist_C(x_1, x_2)\ge 20$, then $V(F_1)\cap V(F_2)=\emptyset$.
    \end{lemma}
    \begin{proof}
    Suppose $\dist_C(x_1, x_2)\ge 20$.
    Let $F_1=v_0v_1v_2 \cdots v_{m_1}v_0$ and $F_2=w_0w_1w_2 \cdots w_{m_2}w_0$ with $x_1=v_0$ and $x_2=w_0$
    such that no vertex in $\{v_2, v_3, v_4\}$ has a neighbor in $V(F_1)\setminus \{v_2, v_3, v_4\}$, and 
    no vertex in  $\{w_2, w_3, w_4\}$ has a neighbor in $V(F_2)\setminus \{w_2, w_3, w_4\}$.

    Towards a contradiction, suppose that $V(F_1)\cap V(F_2)\neq \emptyset$.
    We choose a minimum integer $i$ such that $v_i$ has a neighbor in $F_2$.
    Let $j$ be the minimum integer such that $v_i$ is adjacent to $w_j$.
    Then $X=v_0v_1 \cdots v_iw_jw_{j-1} \cdots w_0$ is a $C$-path of length $i+j+1$, 
    and $X-\{v_0, w_0\}$ is an induced path.
 
    Since $\dist_C(x_1, x_2)\ge 20$, by Lemma~\ref{lem:distancebasic3}, 
    we have $i+j+1\ge 8$, and $i$ or $j$ must be at least $4$. 
    So, $X$ includes one of $\{v_2, v_3, v_4\}$ or $\{w_2, w_3, w_4\}$, 
    which implies that $X$ is an $\h$-path extending $\h$.
    This contradicts the maximality of $\h$.
    \end{proof}
    Using Lemma~\ref{lem:narrow}, we can find either $k$ vertex-disjoint appendages or a set hitting all appendages.

	\begin{lemma}\label{lem:narrowalgo}
	In polynomial time, we can find either $k$ vertex-disjoint appendages of $\h$ or  a vertex set $X_{\ref{lem:narrowalgo}}$ of size at most $20k$ hitting all appendages.
	\end{lemma}
	\begin{proof}
	Initialize $S=\emptyset$.
	For each $v\in V(C)\setminus B$, we test whether there is an appendage with tip $v$. 
	For this, we guess a path $p_mp_0p_1 \cdots p_4$ in $G-(V(\h)\setminus \{v\})$ 
	such that $p_m$ has no neighbor in $\{p_1, \ldots, p_4\}$, $p_1p_2p_3p_4$ is induced, and every vertex in $\{p_2, p_3, p_4\}$ has no neighbor in $\h$.
	For such a path, 
	we test whether there is a path from $p_m$ to $p_4$ in 
	\[G-  ((V(\h)\cup N_G[\{p_1, p_2, p_3\}]) \setminus \{p_4\}).\]
	If there is a path, then by finding a shortest path, 
	we can output an appendage with tip $v$.
	Also, in that case, we add $v$ to $S$.  
	
	Let $X_{\ref{lem:narrowalgo}}$ be the final set $S$. By construction, $X_{\ref{lem:narrowalgo}}$ hits all appendages.  Therefore, we may assume
   $\abs{X_{\ref{lem:narrowalgo}}}> 20k$.  Since $\abs{X_{\ref{lem:narrowalgo}}} > 20k$, there exists a subset $N$ of $X_{\ref{lem:narrowalgo}}$ such that $|N|=k$ and $\dist_C(x_1,x_2) \geq 20$ for all distinct $x_1,x_2 \in N$.  
	Let $\mathcal F$ be the set of appendages output by the algorithm with a tip in $N$.  
	By Lemma~\ref{lem:narrow}, the appendages in  $\mathcal F$  are vertex-disjoint, so we can output $k$ vertex-disjoint appendages, as required.
	\end{proof}

    \section{Hitting long holes}\label{sec:avoiding}
    In the previous sections, we have defined $X_{\ref{lem:greedypacking}}, X_{\ref{lem:sparsealgo}}$, and $X_{\ref{lem:narrowalgo}}$. 
    Let
    \[X_{\ref{lem:tunnellemma4}}:=X_{\ref{lem:greedypacking}} \cup X_{\ref{lem:sparsealgo}}\cup X_{\ref{lem:narrowalgo}}.\]
    In this section, we complete the proof of Theorem~\ref{thm:main2} by either finding $k$ vertex-disjoint long holes of $G$ or sets $X_{\ref{prop:Davoid}}$ and $X_{\ref{prop:traversing}}$ such that $G- (X_{\ref{lem:tunnellemma4}}\cup X_{\ref{prop:Davoid}}\cup X_{\ref{prop:traversing}})$ has no long holes and $|X_{\ref{lem:tunnellemma4}}\cup X_{\ref{prop:Davoid}}\cup X_{\ref{prop:traversing}}|=\mathcal{O}(k \log k)$.  
    
	We divide long holes into two types, depending on whether a long hole intersects $D$ or not.
	We say that a long hole is \emph{$D$-avoiding} if it contains no vertex of $D$, 
	and \emph{$D$-traversing}, otherwise.
	We first consider $D$-avoiding holes in Subsection~\ref{subsec:avoding}, and consider $D$-traversing holes in Subsection~\ref{subsec:traversing}.
	
	Before doing so, we require a few more definitions related to the cycle $C$. For these definitions, we regard $C:=a_0a_1a_2 \cdots a_ma_0$, as a directed cycle where $a_0$ is directed towards $a_1$.
	For a subpath $Q$ of $C$, let $\bd_L(Q)$ and $\bd_R(Q)$ be the ends of $Q$ such that 
	$Q$ is directed from $\bd_L(Q)$ to $\bd_R(Q)$. 
	For an integer $i\ge 2$, 
	let $\bd^i_L(Q):=N^{i-1}_Q[ \bd_L(Q)]$, 
	$\bd^i_R(Q):=N^{i-1}_Q[ \bd_R(Q)]$, $\bd^i(Q):=\bd^i_L(Q)\cup \bd^i_R(Q)$ and
	$\interior^i(Q):=V(Q)\setminus \bd^i(Q)$.

   The following two lemmas will be useful to find a long hole.
       \begin{lemma}\label{lem:tunnellemma4}
    Let $Q$ be a path on more than $160$ vertices in $C-X_{\ref{lem:tunnellemma4}}$, and let $Q^*$ be a subpath of $Q$ on $20$ vertices.
    If $P$ is a path from $\s^{3}_{\bd^{20}_L(Q)}$ to $\s^{3}_{\bd^{20}_R(Q)}$ contained in $G[\s^{3}_{Q}]$, then $P$ intersects $\s^3_{Q^*}\setminus \s^3_{Q-V(Q^*)}$.
    \end{lemma}
    \begin{proof}
    By assumption, $P$ intersects $\s^{3}_{\bd^{20}_L(Q)}$ and $\s^{3}_{\bd^{20}_R(Q)}$.
    Thus, we may assume that $Q-V(Q^*)$ has two distinct connected components. Let $Q_1$ and $Q_2$ be the connected components of $Q-V(Q^*)$, where $\bd_L(Q)\in V(Q_1)$.
    Let $p_1$ and $p_2$ be the ends of $P$ such that $p_1\in \s^3_{\bd^{20}_L(Q)}$.
     
    Suppose for a contradiction that $P$ does not contain a vertex of $\s^3_{Q^*}\setminus \s^3_{Q-V(Q^*)}$.
    Observe that $p_1\in \s^3_{Q_1}\setminus \s^3_{Q^*}$ and $p_2\in \s^3_{Q_2}\setminus \s^3_{Q^*}$, and by Lemma~\ref{lem:distancebasic3}, 
    $p_1\notin \s^3_{Q_2}\setminus \s^3_{Q^*}$ and $p_2\notin \s^3_{Q_1}\setminus \s^3_{Q^*}$.
    Thus, there is an edge $uv$ of $P$ where $u\in \s^3_{Q_1}\setminus \s^3_{Q^*}$ and $v\in \s^3_{Q_2}\setminus \s^3_{Q^*}$. But since $\dist_C(V(Q_1), V(Q_2))\ge 20$, this is not possible by Lemma~\ref{lem:distancebasic3}.
        \end{proof}
       \begin{lemma}\label{lem:tunnellemma5}
    Let $Q$ be a path on $60$ vertices in $C-X_{\ref{lem:tunnellemma4}}$.
    If $P_1$ and $P_2$ are two vertex-disjoint paths from $\s^{3}_{\bd^{20}_L(Q)}$ to $\s^{3}_{\bd^{20}_R(Q)}$ contained in $G[\s^{3}_{Q}]-V(Q)$ such that there are no edges between $P_1$ and $P_2$, then $\s^3_{Q}$ contains a long hole.
    \end{lemma}
    \begin{proof}
    Let $A_1:=\bd^{20}_L(Q)$
    and $A_2:=\bd^{20}_R(Q)$.
    We may assume that for each $i, j\in [2]$, $\abs{V(P_i)\cap \s^3_{A_j}}=1$.
    For $i, j\in [2]$, let $p_{i,j}$ be the end of $P_i$ contained in $\s^3_{A_j}$. By Lemma~\ref{lem:distancebasic3}, $p_{i,1}\neq p_{i,2}$ and $P_i$ has length at least $2$. 	
	For each $i, j\in [2]$, let $P_{i,j}$ be a path of length at most $3$ in $\s^3_{A_j}$ from $p_{i,j}$ to $A_j$.

	For each $j\in [2]$, let $U_j$ be a shortest path from $p_{1,j}$ to $p_{2,j}$ in $G[V(P_{1,j}) \cup V(P_{2,j})\cup A_j]$.
    By the choice of $P_1$ and $P_2$, no vertex of $(P_1-p_{1,1})\cup (P_2-p_{2,1})$ is contained in $\s^3_{A_1}$. This implies that no internal vertex of $U_1$ has  a neighbor in $(P_1-p_{1,1})\cup (P_2-p_{2,1})$.
    Similarly, no internal vertex of $U_2$ has  a neighbor in $(P_1-p_{1,2})\cup (P_2-p_{2,2})$.
    Also, by Lemma~\ref{lem:distancebasic3}, there is no edge between $U_1$ and $U_2$.
    Since $P_1$ and $P_2$ have length at least $2$, $U_1\cup U_2\cup P_1\cup P_2$ is a long hole contained in $\s^3_{Q}$.
    \end{proof} 
    

\subsection{$D$-avoiding long holes}\label{subsec:avoding}

   In this section, we show that by taking at most 380 vertices in each connected component of $C-X_{\ref{lem:tunnellemma4}}$, we can hit all  $D$-avoiding long holes.
    
    We begin with the following structural property.
        
    \begin{lemma}\label{lem:longout}
    The graph $G-(X_{\ref{lem:tunnellemma4}} \cup D)$ has no induced path $P:=p_1p_2 \cdots p_\ell$ of length at least $4$ such that 
    \begin{itemize}
        \item $p_1\in \s^2_q$ for some $q\in V(C)\setminus X_{\ref{lem:tunnellemma4}}$,
        \item $p_2,p_3,p_4\notin \s^2_C$,
        \item $\{p_i:i\in [\ell-1]\}\cap V(\h)=\emptyset$, and 
        \item $p_\ell\in V(\h)$.
    \end{itemize}
    \end{lemma}    
    \begin{proof}
    Suppose for a contradiction that such an induced path $P$ exists.
    Since $p_1\in \s^2_q$ and $p_2\notin \s^2_C$, 
    $p_1$ has no neighbor in $C$.
    Choose a vertex $q'$ in $\s^2_q$ such that $qq'p_1$ is a path.
    
    We claim that no vertex in $\{p_2,p_3,p_4\}$ has a neighbor in $\h$.
    Suppose not. Since $p_2,p_3,p_4$ are not in $\s^1_C$, one of $\{p_2,p_3,p_4\}$ has a neighbor in $V(\h)\setminus V(C)$.
    Thus, there is an $\h$-path from $q$ to a vertex in $V(\h)\setminus V(C)$ that has length at most $5$, which contradicts Lemma~\ref{lem:extension}.

    We observe that 
    $R=qq'p_1Pp_\ell$
    is either an extension or an appendage of $\h$. 
    Assume that $q'$ has a neighbor in $P-\{p_1, p_\ell\}$. 
    Choose a neighbor $z$ in $P-\{p_1, p_\ell\}$ such that $\dist_{P}(p_1, z)$ is minimum.
    As $p_2,p_3,p_4\notin \s^2_C$, $q'$ has no neighbor in $\{p_2, p_3, p_4\}$.
    Thus, $\dist_{P}(p_1, z)\ge 4$. This implies that the cycle $q'p_1Pzq'$ is a long hole in $G-V(C)$, a contradiction.
    Therefore, $q'$ has no neighbor in $P-\{p_1, p_\ell\}$ and in particular, $R-\{q, p_\ell\}$ is induced.
    
    As no vertex in $\{p_2, p_3, p_4\}$ has a neighbor in $\h$, 
    we conclude that $R$ is an extension or an appendage of $\h$.
    This contradicts that $\h$ is maximal and $X_{\ref{lem:tunnellemma4}}$ hits all appendages.
    \end{proof}

        \begin{figure}
  \centering
  \begin{tikzpicture}[scale=0.7]
  \tikzstyle{w}=[circle,draw,fill=black!50,inner sep=0pt,minimum width=4pt]

   \draw (-2,0)--(11,0);
	\draw(-2, 0)--(-3,-0.5);
	\draw(11, 0)--(12,-0.5);
    \draw[dashed](13, -1)--(12,-0.5);
	\draw[dashed](-4,-1)--(-3,-0.5);

    \draw (9, 0) node [w] (v) {};
     \node at (9, -.5) {$v$};
     \node at (2, -.5) {$L_1$};
     \node at (6, -.5) {$L_2$};
     \node at (4.5, -1.7) {$M$};
    \node at (4.5, -2.7) {$Q$};
 
     \node at (-0.7, 0.3) {$0$};
     \node at (-0.7+2, 0.3) {$20$};
     \node at (-0.7+4, 0.3) {$40$};
     \node at (-0.7+6, 0.3) {$60$};
     \node at (-0.7+8, 0.3) {$80$};
    
     \draw (-1, -0.5)--(-1, 3.5);
     \draw (-1.1, -0.5)--(-1.1, 3.5);

     \draw (1, -0.5)--(1, 3.5);
     \draw (3, -0.5)--(3, 3.5);
     \draw (5, -0.5)--(5, 3.5);
     \draw (7, -0.5)--(7, 3.5);

   \draw (1, -1)--(1, -1.4);
   \draw (1, -1.2)--(8, -1.2);

   \draw (-1, -2)--(-1, -2.4);
   \draw (-1.1, -2)--(-1.1, -2.4);
   \draw (-1, -2.2)--(10, -2.2);
   \draw[dashed] (10, -2.2)--(11, -2.2);

    \draw[decorate, decoration={snake}] (-1, 3)--(10, 3);
 	

\draw[rounded corners] (-3-.5+7,4+1)--(-3-.5+7,3+1)--(0-.5+7,3+1)--(0-.5+7,5+1)--(-3-.5+7,5+1)--(-3-.5+7,4+1);
	     \node at (7, 5.5) {$D$};
        \node at (10.5, 1.5) {$P$};
     
    	 \draw (8.2, 1) node [w] (v1) {};
   	 \draw (8.4, 0.5) node [w] (v2) {};

   	 \draw (9, 1) node [w] (v3) {};
   	 \draw (9.5, 2) node [w] (v4) {};
   	 \draw (10, 2) node [w] (v5) {};
   	 \draw (10.5, 2.5) node [w] (v6) {};
   
    	 \draw (2, 1.5) node [w] (w1) {};
    	 \draw (4, 1.5) node [w] (w2) {};
    	 \draw (6, 1.5) node [w] (w3) {};

     	\draw(w1)--(w2)--(w3)--(v1)--(v2)--(v)--(v3)--(v4)--(v5)--(v6);
	\draw[dashed] (w1)--(1.5, 1.8);
        \draw[dashed] (v6)--(11, 2.5);
  \end{tikzpicture}     \caption{The setting in Lemma~\ref{lem:tunnellemma3}. We show that the ends of $P$ have to be contained in  distinct sets of $\s^{3}_{L_1}$ and $\s^{3}_{R_1}$. }\label{fig:tunnel1}
\end{figure}
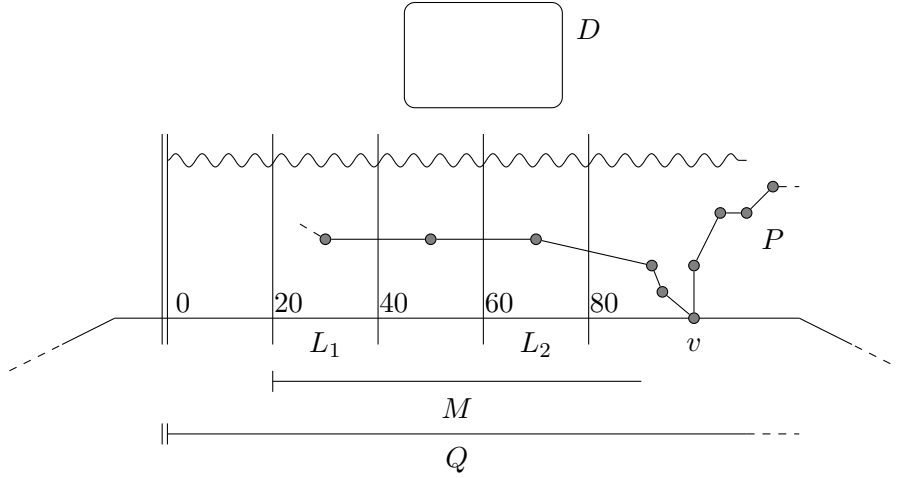
    
    \begin{lemma}\label{lem:tunnellemma3}
    Let $Q$ be a path on more than $160$ vertices in $C-X_{\ref{lem:tunnellemma4}}$.
    Let $H$ be a $D$-avoiding long hole in $G-X_{\ref{lem:tunnellemma4}}$ such that 
  $H$ contains no vertices in $\bd^{80}(Q)$, and it contains a vertex $v$ in $\interior^{80}(Q)$.
    Then the connected component of $H\cap G[\s^{3}_{\interior^{20}(Q)}]$ containing $v$ is a path from $\s^{3}_{\bd^{40}_L(Q)\setminus \bd^{20}_L(Q)}$ to $\s^{3}_{\bd^{40}_R(Q)\setminus \bd^{20}_R(Q)}$.
    \end{lemma}
    \begin{proof}
    Because $X_{\ref{lem:greedypacking}}\subseteq X_{\ref{lem:tunnellemma4}}$,
    by Lemma~\ref{lem:largesupp},
	$H$ is not contained in $\s^3_Q$.
	Let $P$ be the connected component of $H\cap G[\s^{3}_{\interior^{20}(Q)}]$ containing $v$.
	Then $P$ is a path.
	Let $\overline{Q}$ be the other subpath of $C$ with the same ends as $Q$.
	For convenience, 
	let $L_1:=\bd^{40}_L(Q)\setminus \bd^{20}_L(Q)$, 
	$L_2:=\bd^{80}_L(Q)\setminus \bd^{60}_L(Q)$, 
	$R_1:=\bd^{40}_R(Q)\setminus \bd^{20}_R(Q)$,  
	$R_2:=\bd^{80}_R(Q)\setminus \bd^{60}_R(Q)$, and 
	$M:=\interior^{20}(Q)$.
	See Figure~\ref{fig:tunnel1}.
	
We claim that both ends of $P$ are contained in $\s^3_{L_1\cup R_1}$.
Suppose the contrary, and let $w$ be an end of $P$ such that $w\in \s^3_M\setminus \s^3_{L_1\cup R_1}$. 
Let $w'$ be a neighbor of $w$ in $H$ that is not in $P$.
First assume that $w'\in \s^3_{C}$. Since $w'\notin \s^3_M$, 
$w'$ is contained in $\s^3_{\overline{Q}}\cup \s^3_{\bd^{20}(Q)}$. Therefore, there is a $C$-path of length at most $7$ in $G-D$ with one end in $V(\overline{Q})\cup \bd^{20}(Q)$ and the other end in $\interior^{40}(Q)$.
However, since $\dist_C( V(\overline{Q})\cup \bd^{20}(Q), \interior^{40}(Q))\ge 20$, we contradict Lemma~\ref{lem:distancebasic3}.
Thus, we may assume that $w'\notin \s^3_{C}$.

	Since $w' \notin \s^3_{C}$ and $w \in \s^3_{C}$, we have $w\in \s^3_{C}\setminus \s^2_{C}$.
	Following the direction from $w$ to $v$ in $P$, let $a$ be the first vertex so that $a\in \s^2_{C}$.
	Since $a$ is in $\s^3_M$, 
	again by Lemma~\ref{lem:distancebasic3}, 
	$a \notin \s^3_{\overline{Q}}$.
	Thus, $a\in \s^2_q$ for some $q\in V(Q)$. Observe that $q\notin X_{\ref{lem:tunnellemma4}}$. Let $w''$ be the neighbor of $w'$ in $C$ that is not $w$.

	Now, in the path $aPww'w''$, $a\in \s^2_q$ while the next three vertices are not contained in $\s^2_{C}$.
	By Lemma~\ref{lem:extension}, no vertex in $aPww'w''$ is contained in $\h$. Following the direction from $w$ to $w'$ in $H$, let $u$ be the first vertex contained in $\h$. Then the subpath of $H$ from $a$ to $u$ containing $w$ satisfies the conditions of Lemma~\ref{lem:longout}, which yields a contradiction.
	So, we conclude that both ends of $P$ are contained in $\s^3_{L_1\cup R_1}$.

	

      \begin{figure}
  \centering
  \begin{tikzpicture}[scale=0.7]
  \tikzstyle{w}=[circle,draw,fill=black!50,inner sep=0pt,minimum width=4pt]

   \draw (-2,0)--(11,0);
	\draw(-2, 0)--(-3,-0.5);
	\draw(11, 0)--(12,-0.5);
    \draw[dashed](13, -1)--(12,-0.5);
	\draw[dashed](-4,-1)--(-3,-0.5);

    \draw (9, 0) node [w] (v) {};
     \node at (9, -.5) {$v$};
     \node at (2, -.5) {$L_1$};
     \node at (6, -.5) {$L_2$};
     \node at (4.5, -1.7) {$M$};
    \node at (4.5, -2.7) {$Q$};
 
     \node at (-0.7, 0.3) {$0$};
     \node at (-0.7+2, 0.3) {$20$};
     \node at (-0.7+4, 0.3) {$40$};
     \node at (-0.7+6, 0.3) {$60$};
     \node at (-0.7+8, 0.3) {$80$};
    
     \draw (-1, -0.5)--(-1, 3.5);
     \draw (-1.1, -0.5)--(-1.1, 3.5);

     \draw (1, -0.5)--(1, 3.5);
     \draw (3, -0.5)--(3, 3.5);
     \draw (5, -0.5)--(5, 3.5);
     \draw (7, -0.5)--(7, 3.5);

   \draw (1, -1)--(1, -1.4);
   \draw (1, -1.2)--(8, -1.2);

   \draw (-1, -2)--(-1, -2.4);
   \draw (-1.1, -2)--(-1.1, -2.4);
   \draw (-1, -2.2)--(10, -2.2);
   \draw[dashed] (10, -2.2)--(11, -2.2);

    \draw[decorate, decoration={snake}] (-1, 3)--(10, 3);
 	

\draw[rounded corners] (-3-.5+7,4+1)--(-3-.5+7,3+1)--(0-.5+7,3+1)--(0-.5+7,5+1)--(-3-.5+7,5+1)--(-3-.5+7,4+1);
	     \node at (7, 5.5) {$D$};
        \node at (10.5, 1.5) {$P$};
     
    	 \draw (8.2, 1) node [w] (v1) {};
   	 \draw (8.4, 0.5) node [w] (v2) {};

   	 \draw (9, 1) node [w] (v3) {};
   	
	 \draw (8, 1.8) node [w] (a1) {};
   	
	 \draw (6, 2.3) node [w] (v4) {};
   	 \draw (4, 2.3) node [w] (v5) {};
   	 \draw (2, 2.3) node [w] (v6) {};
   
    	 \draw (2, 1.5) node [w] (w1) {};
    	 \draw (4, 1.5) node [w] (w2) {};
    	 \draw (6, 1.5) node [w] (w3) {};

    	 \draw (2.3, 1.2) node [w] (z1) {};
    	 \draw (2.3, 0.6) node [w] (z2) {};
    	 \draw (2.3, 0) node [w] (z3) {};
    	 \draw (1.7, 1.2) node [w] (z4) {};
    	 \draw (1.7, 0.6) node [w] (z5) {};
    	 \draw (1.7, 0) node [w] (z6) {};

	\draw[very thick] (w1)--(z1)--(z2)--(z3);
     \draw[very thick] (v6)--(z4)--(z5)--(z6);

    	 \draw (6.3, 1.2) node [w] (b1) {};
    	 \draw (6.3, 0.6) node [w] (b2) {};
    	 \draw (6.3, 0) node [w] (b3) {};
    	 \draw (5.7, 1.2) node [w] (b4) {};
    	 \draw (5.7, 0.6) node [w] (b5) {};
    	 \draw (5.7, 0) node [w] (b6) {};

	\draw[very thick] (w3)--(b1)--(b2)--(b3);
     \draw[very thick] (v4)--(b4)--(b5)--(b6);

     	\draw(w1)--(w2)--(w3)--(v1)--(v2)--(v)--(v3)--(a1)--(v4)--(v5)--(v6);
  \end{tikzpicture}     \caption{The construction of a long hole in Lemma~\ref{lem:tunnellemma3} when the two ends of $P$ are contained in $\s^3_{L_1}$. }\label{fig:tunnel2}
\end{figure}
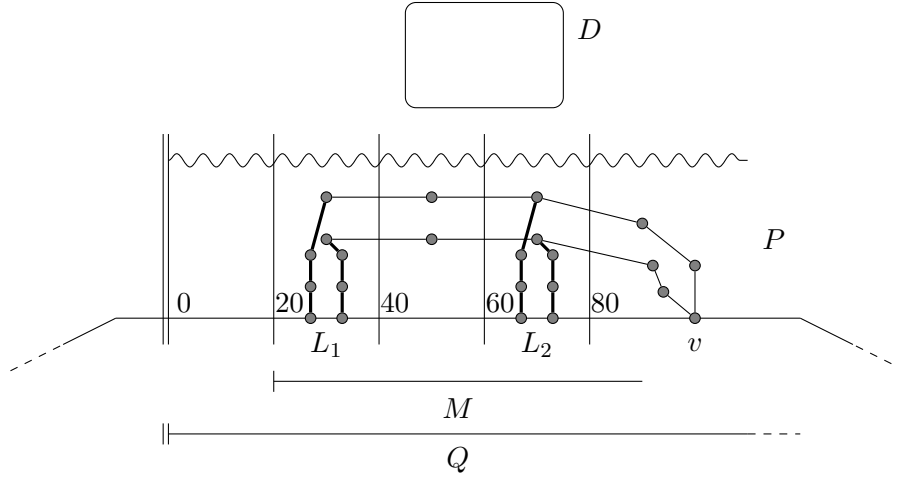

	To conclude the lemma, we claim that one end of $P$ is in $\s^3_{L_1}$ and the other end is in $\s^3_{R_1}$. Suppose not.  By symmetry, we may assume that both ends of $P$ are in $\s^3_{L_1}$.
	See Figure~\ref{fig:tunnel2} for an illustration.

	Let $P_1$ and $P_2$ be the two subpaths of $P$ from $v$ to the ends of $P$.
	Applying Lemma~\ref{lem:tunnellemma4} to $P_1$ and $P_2$, we deduce that for each $i\in [2]$, 
	$P_i$ contains a vertex of $\s^3_{L_2}$.
	For each $i\in [2]$, let $Q_i$ be a shortest subpath of $P_i$ from $\s^3_{L_1}$ to $\s^3_{L_2}$ respectively. By Lemma~\ref{lem:distancebasic3}, all the internal vertices of $Q_i$ are contained in $\s^3_{\bd^{60}_L(Q) \setminus \bd^{40}_L(Q)}$, and $Q_i$ is contained in $\s^3_{\bd^{80}_L(Q) \setminus \bd^{20}_L(Q)}$.
	 By Lemma~\ref{lem:tunnellemma5}, $G\left[\s^3_{\bd^{80}_L(Q) \setminus \bd^{20}_L(Q)}\right]$ contains a long hole.
	 This contradicts Lemma~\ref{lem:largesupp}.
		\end{proof}

	

	

   \begin{proposition}\label{prop:Davoid}
There is a polynomial-time algorithm to find a vertex set $X_{\ref{prop:Davoid}}\subseteq V(G)\setminus X_{\ref{lem:tunnellemma4}}$ of size at most $380 \abs{V(C)\cap X_{\ref{lem:tunnellemma4}}}$ such that $X_{\ref{lem:tunnellemma4}}\cup X_{\ref{prop:Davoid}}$ hits all $D$-avoiding long holes.  
\end{proposition}
\begin{proof}
Recall that $C=a_0a_1a_2 \cdots a_ma_0$ and consider it as a directed cycle where $a_0$ is directed to $a_1$.
For a subpath $Q$ of $C$ and two vertices $a_i$ and $a_j$ in $Q$, 
we denote by $a_i\preccurlyeq_Q a_j$ if there is a directed path from $a_i$ to $a_j$.

We construct a set $S$ as follows.
We iterate the following over all components $Q$ of $C-X_{\ref{lem:tunnellemma4}}$.
For each $i\in [7]$, let $L_i:=\bd^{20(i+1)}_L(Q)\setminus \bd^{20i}_L(Q)$, $R_i:=\bd^{20(i+1)}_R(Q)\setminus \bd^{20i}_R(Q)$, and $M:=\interior^{20}(Q)$.
\begin{enumerate}[(1)] 
\item If $Q$ has at most $320$ vertices, then we add all vertices of $Q$ to $S$.
\item Assume that $Q$ has more than $320$ vertices. We first add $\bd^{160}(Q)$ to $S$.
Let $F_L$ be the connected component of $G[\s^3_{M}]-M$ intersecting both $\s^3_{L_1}$ and $\s^3_{L_3}$, if one exists. Similarly, let $F_R$ be the connected component of $G[\s^3_M]-M$ intersecting both $\s^3_{R_1}$ and $\s^3_{R_3}$, if one exists. If $F_L$ or $F_R$ does not exist, then we skip this component $Q$. (If there are two components intersecting both $\s^3_{L_1}$ and $\s^3_{L_2}$, then there is a long hole contained in $\s^3_{L_1\cup L_2\cup L_3}$ by Lemma~\ref{lem:tunnellemma5}, which contradicts Lemma~\ref{lem:largesupp}. Thus, if such a component exists, then the component uniquely exists, and the same argument holds for $\s^3_{R_1}$ and $\s^3_{R_3}$.)
\item Now, assume that both $F_L$ and $F_R$ exist. If $F_L=F_R$, then we skip this component $Q$. If $F_L\neq F_R$, then 
let $q$ be the vertex of $Q$ where $\s^3_{q}\cap V(F_L)\neq \emptyset$ and it is largest with respect to the $\preccurlyeq_Q$ relation. 
We add the set $\{v\in V(Q) : v\preccurlyeq_Q q, \dist_Q(v, q)\le 59 \}$ to $S$.
\end{enumerate}

Let $X_{\ref{prop:Davoid}}$ be the final set $S$.
Clearly, $\abs{X_{\ref{prop:Davoid}}}\le 380\abs{V(C)\cap X_{\ref{lem:tunnellemma4}}}$ and it can be computed in polynomial time.

It remains to show that $G-(X_{\ref{lem:tunnellemma4}}\cup X_{\ref{prop:Davoid}})$ has no $D$-avoiding long hole. Suppose that 
there is a $D$-avoiding long hole $H$ in $G-(X_{\ref{lem:tunnellemma4}}\cup X_{\ref{prop:Davoid}})$. We choose such a hole with minimum $\abs{V(H)\cap V(C)}$. Since $G-V(C)$ has no long hole, 
$H$ must contain a vertex of $C$. Let $Q$ be a connected component of $C-X_{\ref{lem:tunnellemma4}}$ containing a vertex of $H$, say $v$.
Observe that when we constructed $X_{\ref{prop:Davoid}}$, we added $\bd^{160}(Q)$ to $X_{\ref{prop:Davoid}}$, 
and therefore, $v$ is contained in $\interior^{160}(Q)$.
So, by Lemma~\ref{lem:tunnellemma3}, the connected component of $H\cap G[\s^{3}_M]$ containing $v$ is a path from $\s^{3}_{L_1}$ to $\s^{3}_{R_1}$. Let $P$ be the path.
Clearly its ends are contained in $F_L$ and $F_R$, respectively. 

We separately consider the two cases depending on whether $F_L=F_R$ or not.

\medskip
\noindent {\bf Case 1.} $F_L\neq F_R$.

As $\interior^{160}(Q)$ is non-empty, $Q$ has more than $320$ vertices. Let $q$ be the vertex of $Q$ where $\s^3_{q}\cap V(F_L)\neq \emptyset$ and it is largest with respect to the $\preccurlyeq_Q$ relation. 
By the construction of $X_{\ref{prop:Davoid}}$, the set  $\{v\in V(Q) : v\preccurlyeq_Q q, \dist_Q(v, q)\le 59 \}$ was added to $X_{\ref{prop:Davoid}}$.
Let $U:=\{v\in V(Q) : v\preccurlyeq_Q q, \dist_Q(v, q)\le 59 \}$.

Observe that if $q$ is contained in $\bd^{100}_R(Q)$, then by Lemma~\ref{lem:tunnellemma4}, each of $F_L$ and $F_R$ contains a path from $\bd^{100}_{R_5}(Q)$ to $\bd^{100}_{R_7}(Q)$, and by Lemma~\ref{lem:tunnellemma5},
there is a long hole contained in $\bd^{100}_{R_5\cup R_6\cup R_7}(Q)$, a contradiction.
So, $q$ is not contained in $\bd^{100}_R(Q)$.
Also, the path $P$ is a path from $\s^3_{L_1}$ to $\s^3_{R_1}$, 
$q$ is not contained in $\bd^{160}_L(Q)$.
Therefore, $Q-U$ has two connected components.
Let $Q_1$ and $Q_2$ be the connected components of $Q-U$ such that $\bd^{20}(Q)\in V(Q_1)$.

As $P$ is a path from $\s^{3}_{L_1}$ to $\s^{3}_{R_1}$, 
$P$ is a path between $\s^3_{Q_1}$ and $\s^3_{Q_2}$. Let $P'$ be a shortest subpath of $P$ from $\s^3_{Q_1}$ to $\s^3_{Q_2}$. 
By the minimality, $P'$ does not contain a vertex of $Q$, because every neighbor of a vertex in $Q_1$ is in $\s^3_{Q_1}$ and every neighbor of a vertex in $Q_2$ is in $\s^3_{Q_2}$.
Thus, $P'$ is fully contained in $G[\s^3_{M}]-M$.
On the other hand, since $\s^3_{Q_2}\cap V(F_L)=\emptyset$ while $P'$ intersects $\s^3_{Q_2}$, 
$P'$ is not contained in $F_L$.
So, each of $P'$ and $F_L$
contains a path from $\s^{3}_{\bd^{20}_L(G[U])}$ to 
$\s^{3}_{\bd^{20}_R(G[U])}$.
 By Lemma~\ref{lem:tunnellemma5}, there is a long hole contained in $\s^3_U$, which contradicts the consequence of Lemma~\ref{lem:largesupp}.

\medskip
\noindent {\bf Case 2.} $F_L= F_R$.

By Lemma~\ref{lem:tunnellemma4}, $P$ intersects $\s^3_{L_7}\cap V(F_L)$ and $\s^3_{R_7} \cap V(F_L)$.
We choose a vertex $w_1$ of $P$ in $\s^3_{L_7}\cap V(F_L)$ and 
$w_2$ of $P$ in $\s^3_{R_7} \cap V(F_L)$. Let $Z$ be a shortest path from $w_1$ to $w_2$ in $F_L$.
Let $v_1v_2v_3v_4v_5$ be a subpath of $H$ such that $v_3\in \s^3_{L_2}\setminus \s^3_{L_1}$ and $v_3\in V(P)$. Such a vertex $v_3$ exists by Lemma~\ref{lem:tunnellemma4}.
Then $\{v_2, v_3, v_4\}$ is contained in $\s^3_{L_1\cup L_2\cup L_3}$.

If $Z$ contains a vertex in $\s^3_{\bd^{100}_L(Q)}$, then 
there are two subpaths of $P$ from $\s^3_{L_5}$ to $\s^3_{L_7}$ contained in $\s^3_{L_5\cup L_6\cup L_7}$.
Then by Lemma~\ref{lem:tunnellemma5}, there is a long hole contained in $\s^3_{L_5\cup L_6\cup L_7}$, a contradiction.
Therefore, $Z$ contains no vertex in $\s^3_{\bd^{100}_L(Q)}$, and by Lemma~\ref{lem:distancebasic3}, 
no vertex in $\{v_2, v_3, v_4\}$ has a neighbor in $Z$.

Let $H'$ be the path from $v_1$ to $v_5$ that does not contain $v_2$. Let $H'':=(H'- (V(P)\setminus \{w_1, w_2\}))\cup Z$. Clearly, $H''$ is connected. Let $H'''$ be a shortest path from $v_1$ to $v_5$ in $H''$.
Observe that there are no edges between $\{v_2, v_3, v_4\}$ and $V(H''')\setminus \{v_1, v_5\}$ in $G$. So, by Lemma~\ref{lem:detecting}, $v_1v_2v_3v_4v_5(H''')v_1$ is a $D$-avoiding long hole. As $v\in V(P)$ and $v\notin Z$, this long hole has vertices of $C$ strictly less than $H$. This contradicts the minimality of $\abs{V(H)\cap V(C)}$.

\medskip

We conclude that $G-(X_{\ref{lem:tunnellemma4}}\cup X_{\ref{prop:Davoid}})$ has no $D$-avoiding long hole.
\end{proof}

    \subsection{$D$-traversing long holes}\label{subsec:traversing}
    
    It remains to deal with $D$-traversing long holes. 
    We begin by constructing a second greedy packing of long holes. 
    This step is similar to the first greedy packing in Section~\ref{sec:greedy}, but now we also consider the vertex set $V(G)\setminus (D \cup \s^3_C$). 
    

    
    \medskip
\textbf{Algorithm. (Second greedy packing)}

\begin{enumerate}
	\item Set $G_1:=G-(X_{\ref{lem:tunnellemma4}}\cup X_{\ref{prop:Davoid}})$ and 
    initialize $C_1=C-(X_{\ref{lem:tunnellemma4}}\cup X_{\ref{prop:Davoid}})$, $D_1=D\setminus (X_{\ref{lem:tunnellemma4}}\cup X_{\ref{prop:Davoid}})$, and $M_1=\emptyset$. 
    \item Choose a previously unchosen pair $(A,v)$, where $A$ is a non-empty subset of size at most $2$ of $D_i$ and $v$ is a vertex of $C_i$.  If there are no remaining pairs, then proceed to step (6).
    \item Test if $G_1-M_i$ contains a long hole such that it contains $v$ and its intersection on $D_i$ is contained in $A$.   
    \item If no, then proceed to step (1).  
    \item If yes, then set $D_{i+1}:=D_i \setminus A$ and $C_{i+1}:=C_i - N^{90}_C[v]$ and $M_{i+1}:=M_i\cup A\cup N^{90}_C[v]$, then proceed to step (2).  
    \item Let $\ell$ be the largest index for which $D_{\ell}$ exists. Define $X_{\ref{prop:traversing}}:=M_{\ell}$.
\end{enumerate}


    \begin{figure}
  \centering
  \begin{tikzpicture}[scale=0.7]
  \tikzstyle{w}=[circle,draw,fill=black!50,inner sep=0pt,minimum width=4pt]

   \draw (-2,0)--(11,0);
	\draw(-2, 0)--(-3,-0.5);
	\draw(11, 0)--(12,-0.5);
    \draw[dashed](13, -1)--(12,-0.5);
	\draw[dashed](-4,-1)--(-3,-0.5);

    \draw (4.5, 0) node [w] (v) {};
     \node at (4.5, -.5) {$v$};
     \node at (4.5, -1.7) {$N^{15}_C[v]$};
    \node at (4.5, -2.7) {$N^{35}_C[v]$};
    
     \draw (1, -0.5)--(1, 3.5);

     \draw (-1, -0.5)--(-1, 3.5);
     \draw (-1.1, -0.5)--(-1.1, 3.5);

     \draw (8, -0.5)--(8, 3.5);

     \draw (10, -0.5)--(10, 3.5);
     \draw (10.1, -0.5)--(10.1, 3.5);

   \draw (1, -1)--(1, -1.4);
   \draw (8, -1)--(8, -1.4);
   \draw (1, -1.2)--(8, -1.2);

   \draw (-1, -2)--(-1, -2.4);
   \draw (-1.1, -2)--(-1.1, -2.4);
   \draw (10, -2)--(10, -2.4);
   \draw (10.1, -2)--(10.1, -2.4);
   \draw (-1, -2.2)--(10, -2.2);

    \draw[decorate, decoration={snake}] (-1, 3)--(10, 3);
 	
   \node at (0, 2.3) {$\s^3_{\bd^{20}(Q)}$};
   \node at (2.1, 2.3) {$\s^3_{\interior^{20}(Q)}$};

\draw[rounded corners] (-3-.5+7,4+1)--(-3-.5+7,3+1)--(0-.5+7,3+1)--(0-.5+7,5+1)--(-3-.5+7,5+1)--(-3-.5+7,4+1);
	\draw[rounded corners] (-2-.5+7,3.5+1)--(-2-.5+7,3+1)--(-1-.5+7,3+1)--(-1-.5+7,4+1)--(-2-.5+7,4+1)--(-2-.5+7,3.5+1);
     \node at (6, 4.5) {$A$};
     \node at (7, 5.5) {$D$};
     
    	 \draw (2.2, 1) node [w] (v1) {};
   	 \draw (3.4, 0.5) node [w] (v2) {};

   	 \draw (5, 1) node [w] (v3) {};
   	 \draw (5.5, 2) node [w] (v4) {};
   	 \draw (6, 2) node [w] (v5) {};
   	 \draw (6.5, 2.5) node [w] (v6) {};

   	 \draw (7, 3.5) node [w] (v7) {};
   	 \draw (6, 3.8) node [w] (v8) {};
   	 \draw (5, 4.3) node [w] (v9) {};
     
     	\draw(v1)--(v2)--(v)--(v3)--(v4)--(v5)--(v6)--(v7)--(v8)--(v9)--(v1);
     
  \end{tikzpicture}     \caption{An example of a long hole $H$ in Lemma~\ref{lem:2ndgreedycompo1}, where $Q=G[N^{35}_C[v]]$. }\label{fig:traverse1}
\end{figure}
    
    We first show that when we consider the pair $(A, v)$, if a long hole $H$ is detected, then 
    the connected component of $H\cap G[\s^3_Q]$ containing $v$ is a path, 
    where $Q=G[N^{35}_C[v]]$. See Figure~\ref{fig:traverse1} for an example.

        \begin{figure}
  \centering
  \begin{tikzpicture}[scale=0.7]
  \tikzstyle{w}=[circle,draw,fill=black!50,inner sep=0pt,minimum width=4pt]

   \draw (-2,0)--(11,0);
	\draw(-2, 0)--(-3,-0.5);
	\draw(11, 0)--(12,-0.5);
    \draw[dashed](13, -1)--(12,-0.5);
	\draw[dashed](-4,-1)--(-3,-0.5);

    \draw (4.5, 0) node [w] (v) {};
     \node at (4.5, -.5) {$v$};
     \node at (4.5, -1.7) {$N^{15}_C[v]$};
    \node at (4.5, -2.7) {$N^{35}_C[v]$};
    
     \draw (1, -0.5)--(1, 3.5);

     \draw (-1, -0.5)--(-1, 3.5);
     \draw (-1.1, -0.5)--(-1.1, 3.5);

     \draw (8, -0.5)--(8, 3.5);

     \draw (10, -0.5)--(10, 3.5);
     \draw (10.1, -0.5)--(10.1, 3.5);

   \draw (1, -1)--(1, -1.4);
   \draw (8, -1)--(8, -1.4);
   \draw (1, -1.2)--(8, -1.2);

   \draw (-1, -2)--(-1, -2.4);
   \draw (-1.1, -2)--(-1.1, -2.4);
   \draw (10, -2)--(10, -2.4);
   \draw (10.1, -2)--(10.1, -2.4);
   \draw (-1, -2.2)--(10, -2.2);

    \draw[decorate, decoration={snake}] (-1, 3)--(10, 3);
 	
   \node at (0, 2.3) {$\s^3_{\bd^{20}(Q)}$};
   \node at (2.2, 2.3) {$\s^3_{\interior^{20}(Q)}$};
 
\draw[rounded corners] (-3-.5+7,4+1)--(-3-.5+7,3+1)--(0-.5+7,3+1)--(0-.5+7,5+1)--(-3-.5+7,5+1)--(-3-.5+7,4+1);
	\draw[rounded corners] (-2-.5+7,3.5+1)--(-2-.5+7,3+1)--(-1-.5+7,3+1)--(-1-.5+7,4+1)--(-2-.5+7,4+1)--(-2-.5+7,3.5+1);
     \node at (6, 4.5) {$A$};
     \node at (7, 5.5) {$D$};
     \node at (8.8, 2.7) {$x$};
     \node at (5, 4.7) {$d$};
     
     	\draw(5, 4.3)--(4, 0); \draw(5, 4.3)--(8.8,0);
		 \draw (4, 0) node [w] (b) {};
		 \draw (8.8, 0) node [w] (c) {};
   
   \node at (6, 1) {$R_1$};
   \node at (8.6, 1.3) {$R_2$};
   
    	 \draw (5, 1) node [w] (v3) {};
   	 \draw (5.5, 2) node [w] (v4) {};
   	 \draw (6, 2) node [w] (v5) {};
   	 \draw (7, 2.5) node [w] (v6) {};

   	 \draw (8.5, 2.5) node [w] (v7) {};
   
   	 \draw (9, 2) node [w] (v8) {};
   	 \draw (9.2, 1) node [w] (v9) {};
   	 \draw (9.2, 0) node [w] (v0) {};
 
   		 \draw (5, 4.3) node [w] (A) {};
     
     	\draw (v)--(v3)--(v4)--(v5)--(v6)--(v7)--(v8)--(v9)--(v0);
     
  \end{tikzpicture}     \caption{The paths $Q$ and $R$ in Lemma~\ref{lem:2ndgreedycompo1}. The vertex $d$ may have a neighbor in $R-x$.}\label{fig:traverse2}
\end{figure}
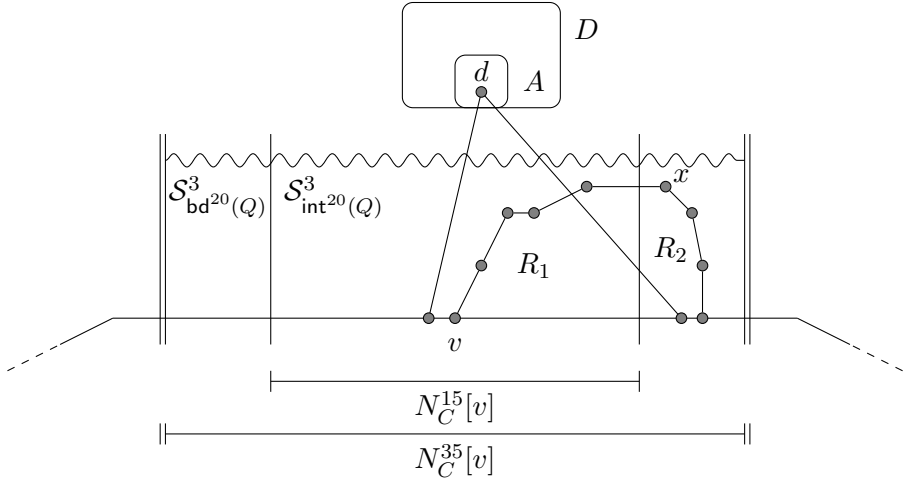

    \begin{lemma}\label{lem:2ndgreedycompo1}
    Let $(A, v)$ be a pair considered in the $i$-th step of the second greedy packing such that $G_1-M_i$ contains a $D$-traversing long hole $H$ with $v\in V(H)$ and $V(H)\cap D_i\subseteq A$.
    Let $Q=G[N^{35}_C[v]]$.
    Then the connected component of $H\cap G[\s^3_Q]$ containing $v$ 
    is a path $P$ such that 
    \begin{itemize}
        \item $V(P)\cap \s^3_{\bd^{20}(Q)}=\emptyset$, and
        \item If $a$ is an end of $P$ and $b$ is the neighbour of $a$ in $H$ such that $b \notin V(P)$, then $b \in V(G)\setminus \s^3_C$.
    \end{itemize}
    \end{lemma}
    \begin{proof}
    Let $\overline{Q}$ be the other subpath of $C$ with the same ends as $Q$.
	Since $H$ is $D$-traversing and $D\cap \s^3_Q=\emptyset$, 
    $P$ is a path.
    
    We first show that $V(P)\cap \s^3_{\bd^{20}(Q)}=\emptyset$. Suppose for a contradiction that $P$ contains a vertex $w$ of $\s^3_{\bd^{20}(Q)}$.
    Following the path $P$ from $v$ to $w$, 
    let $x$ be the first vertex contained in $\s^3_{\bd^{20}(Q)}$.
    Let $R_1=q_1q_2 \cdots q_t$ be the path $vPx$ where $q_1=v$ and $q_t=x$.
    Since $x\in \s^3_{\bd^{20}(Q)}$, there is a path $R_2$ of length at most $3$ from $x$ to $\bd^{20}(Q)$ in $\s^3_{\bd^{20}(Q)}$.
    See Figure~\ref{fig:traverse2} for an illustration.
    Note that no vertex of $(V(R_2)\setminus \{x\}) \cup \bd^{20}(Q)$ has a neighbor in $R_1-x$, because $x$ is the first vertex that is contained in $\s^3_{\bd^{20}(Q)}$.
    
     Let $d\in A\cap V(H)$. Since $d$ is a vertex of $H$, it has no neighbor in the internal vertices of $R_1$.
    As $d$ is almost $C$-dominating, 
    $d$ has a neighbor in $\bd^{20}(Q)$ and has a neighbor in $N_C[v]$.
    Let $T$ be a shortest path from $N_G(d)\cap (\bd^{20}(Q) \cup V(R_2))$ to $N_G(d)\cap N_C[v]$
    in \[  G[\bd^{20}(Q)\cup N_C[v]\cup V(R_1) \cup V(R_2)]. \]
     
    We claim that $T$ contains at least $3$ internal vertices of $R_1$. 
  	Suppose not.
	Then the distance from $N_C[v]$ to $\bd^{20}(Q) \cup V(R_2)$ in $T$ is at most $3$, and 
	$T\cup R_2$ contains a path of length at most $6$ from $N_C[v]$ to $\bd^{20}(Q)$.
	By Lemma~\ref{lem:distancebasic3}, $\dist_C(N_C[v], \bd^{20}(Q))<16$. 
 	However, this contradicts $\dist_C(v, \bd^{20}(Q))\ge 16$. 
	We conclude that $T$ contains at least $3$ internal vertices of $Q$.
   
    Thus, $G[V(T)\cup \{d\}]$ is a long hole contained in $\s^3_Q\cup D$.
    But such a long hole had to be considered during the first greedy packing. 
    This is a contradiction.

    Now, we verify the second statement. Let $a$ be an end of $P$ and $b$ be the neighbour of $a$ in $H$ such that $b \notin V(P)$. 
    Assume for a contradiction that $b\in \s^3_C$.
    Because $P$ is a connected component of $H\cap G[\s^3_Q]$, 
    $b$ is contained in $\s^3_{\overline{Q}}\setminus \s^3_Q$.
    By Lemma~\ref{lem:distancebasic3}, $a$ is contained in $\s^3_{\bd^{20}(Q)}\setminus \s^3_{\interior^{20}(Q)}$. But this is not possible because $V(P)\cap \s^3_{\bd^{20}(Q)}=\emptyset$.
    \end{proof}

      \begin{proposition}\label{prop:traversing}
      Let $\ell$ be the largest index for which $D_\ell$ exists in the second greedy packing algorithm.  Then $G$ contains $\ell-1$ vertex-disjoint long holes and $\abs{X_{\ref{prop:traversing}}}\le 183\ell$.
Moreover, $G-(X_{\ref{lem:tunnellemma4}}\cup X_{\ref{prop:Davoid}}\cup X_{\ref{prop:traversing}})$ has no long hole.
      \end{proposition}
      \begin{proof}
      It is straightforward to verify that $\abs{X_{\ref{prop:traversing}}}\le 183\ell$.
      To prove that $G$ contains $\ell-1$ vertex-disjoint long holes, it suffices to prove that the holes constructed by the second greedy packing are vertex-disjoint.  
Let $H_1$ and $H_2$ be two $D$-traversing long holes constructed in the algorithm, 
and let $(A_1, v_1)$ and $(A_2, v_2)$ be the considered pairs when $H_1$ and $H_2$ are constructed, respectively.
For each $i\in [2]$, let $P_i:=H_i-A_i$, $Q_i:= G[N^{35}_C[v_i]]$, and let $P_i^*$ be the connected component of $H_i\cap G[\s^3_{Q_i}]$ containing $v_i$.
By Step (5) of  the second greedy packing, we have $\dist_C(V(Q_1), V(Q_2))\ge 20$ and by Lemma~\ref{lem:distancebasic3}, $\s^3_{Q_1}$ and $\s^3_{Q_2}$ are vertex-disjoint. In particular, $P_1^*$ and $P_2^*$ are vertex-disjoint.
	
	Suppose that $V(H_1) \cap V(H_2) \neq \emptyset$. Since $A_1\cap A_2=\emptyset$, we have $V(P_1) \cap V(P_2) \neq \emptyset$.
	Let $w_1$ be a vertex of $P_1$ that has a neighbor in $P_2$ such that $\dist_{P_1}(v_1, w_1)$ is minimum, 
	and let $w_2$ be a neighbor of $w_1$ in $P_2$ such that $\dist_{P_2}(v_2, w_2)$ is minimum.
	By the choice of $w_1$ and $w_2$, $v_1P_1w_1w_2P_2v_2$ is an induced path. Let $R:=v_1P_1w_1w_2P_2v_2$.
	
	We claim that 
    for some $i\in [2]$, $R$ contains a subpath of $P_i^*$ from $v_i$ to an end of $P_i^*$ and contains the next two vertices in $H_i$. If not, then by Lemma~\ref{lem:2ndgreedycompo1} there is a path of length at most 3 in $G-D$ between $\s^3_{\interior^{20}(Q_1)}$ and $\s^3_{\interior^{20}(Q_2)}$, and there is a path of length at most $9$ in $G-D$, whose ends are contained in $\interior^{20}(Q_1)$ and $\interior^{20}(Q_2)$, respectively. Since $\dist_C(\interior^{20}(Q_1), \interior^{20}(Q_2))\ge 28$, this contradicts Lemma~\ref{lem:distancebasic3}.
    
    Without loss of generality, $R$ contains a subpath of $P_1^*$ from $v_1$ to an end of $P_1^*$, say $z$, and contains the next two vertices $z_1$ and $z_2$ in $H_1$. Following the direction from $z$ to $z_1$ in $R$, we choose the first vertex $u$ contained in $\h$.
    
    By Lemma~\ref{lem:2ndgreedycompo1}, 
    $z_1$ is contained in $V(G)\setminus \s^3_C$ and $z$ is contained in $\s^3_C\setminus \s^2_C$. Following the direction from $z$ to $v_1$ in $P_1^*$, let $y$ be the first vertex contained in $\s^2_C$. Observe that $yRu$ is an induced path and the next three vertices of $yRu$ after $y$ are not in $\s^2_C$. However, such a path $yRu$ does not exist by Lemma~\ref{lem:longout}.  Thus, the holes constructed by the second greedy packing are vertex-disjoint.      

      It remains to prove that $G-(X_{\ref{lem:tunnellemma4}}\cup X_{\ref{prop:Davoid}}\cup X_{\ref{prop:traversing}})$ has no long hole.  By Proposition~\ref{prop:Davoid}, $G-(X_{\ref{lem:tunnellemma4}}\cup X_{\ref{prop:Davoid}})$ has no $D$-traversing long hole.  Thus, it suffices to prove that $G-(X_{\ref{lem:tunnellemma4}}\cup X_{\ref{prop:Davoid}}\cup X_{\ref{prop:traversing}})$ has no $D$-avoiding long hole. 
      Since $D$ is a clique by Lemma~\ref{lem:cdominating} and $H$ is a hole, $H$ contains at most 2 vertices of $D$, 
      and if it contains two vertices of $D$, then they are adjacent in $H$. 
      So, $H-D$ is a path. Since $G-V(C)$ has no long hole, $H$ contains a vertex of $C$.
      Let $v$ be a vertex of $H$ in $C$.
     By considering the pair $(V(H)\cap D, v)$ in the second greedy packing, we should have proceeded one more step, which is a contradiction.  
     \end{proof}
     
     We are now ready to prove Theorem~\ref{thm:main2}.
     
     \begin{proof}[Proof of Theorem~\ref{thm:main2}]
     We begin by running the first greedy packing to obtain $X_{\ref{lem:greedypacking}}$. We may assume that the largest index $\ell$ in the first greedy packing is at most $k$, else we obtain $k$ vertex-disjoint long holes by Lemma~\ref{lem:greedypacking}. Next, we construct a maximal good sparse ear decomposition $\h$ using Lemma~\ref{lem:sparsealgo}. 
     By Lemma~\ref{lem:narrowalgo}, we can either find $k$ vertex-disjoint appendages of $\h$, or a vertex set $X_{\ref{lem:narrowalgo}}$ of size at most $20k$ hitting all appendages. Since each appendage contains a long hole, we may assume that $X_{\ref{lem:narrowalgo}}$ exists.

     We then apply Proposition~\ref{prop:Davoid} to find $X_{\ref{prop:Davoid}}$ such that $X_{\ref{lem:tunnellemma4}}\cup X_{\ref{prop:Davoid}}$ hits all $D$-avoiding long holes.
     Lastly, we run the second greedy packing to obtain $X_{\ref{prop:traversing}}$.  
     We may assume that the largest index $\ell$ in the second greedy packing is at most $k$, else we are done by Proposition~\ref{prop:traversing}.
     Since each of the above steps is performed in polynomial time, the entire algorithm is a polynomial-time algorithm. 
     
     We claim that we may take $X_{\ref{thm:main2}}= X_{\ref{lem:tunnellemma4}}\cup X_{\ref{prop:Davoid}}\cup X_{\ref{prop:traversing}}$.  By Proposition~\ref{prop:traversing}, $X_{\ref{thm:main2}}$ hits all long holes of $G$, so   
     it only remains to bound the size of $X_{\ref{thm:main2}}$.  
     Recall that $X_{\ref{lem:tunnellemma4}}=X_{\ref{lem:greedypacking}} \cup X_{\ref{lem:sparsealgo}}\cup X_{\ref{lem:narrowalgo}}$ and $X_{\ref{lem:sparsealgo}}=N^{\buff}_C[B\cap V(C)] \cup (B\setminus V(C))$, where $B$ is the set of branching points of $\h$. By Lemma~\ref{lem:narrowalgo}, $|X_{\ref{lem:narrowalgo}}| \leq 20k$ .  If $\h$ has at least $s_k$ branching points, then by Lemma~\ref{lem:cycletohole}, we can output $k$ vertex-disjoint long holes. Thus, we may assume that $|B| < s_k$. Therefore, 
     \[
     |X_{\ref{lem:sparsealgo}}|=|N^{\buff}_C[B\cap V(C)] \cup (B\setminus V(C))| \leq 63s_k.
     \]
     By Lemma~\ref{lem:greedypacking}, $\abs{X_{\ref{lem:greedypacking}}} \leq 212k$, and  by Proposition~\ref{prop:traversing},  $\abs{X_{\ref{prop:traversing}}} \leq 183k$.  Finally, by Proposition~\ref{prop:Davoid}, $\abs{X_{\ref{prop:Davoid}}} \leq 380\abs{X_{\ref{lem:tunnellemma4}}}$.  Putting this all together,

     \[\abs{X_{\ref{thm:main2}}}=\abs{X_{\ref{lem:tunnellemma4}}\cup X_{\ref{prop:Davoid}}\cup X_{\ref{prop:traversing}}}\le 381(212k+63s_k+20k)+183k = \mu_k. \qedhere\] 
     \end{proof}

 \section{Open Problems} \label{sec:openproblems}
 In this paper, we proved that holes of length at least $6$ have the \EP{} property in $C_4$-free graphs.  Our proof is also the shortest known proof that holes have the \EP{} property.  Our first open problem is to answer Question~\ref{ques:EPlong} for larger values of $\ell$. By making our proof much longer, we believe our methods could answer Question~\ref{ques:EPlong} affirmatively for $\ell=7$, but new ideas are needed for $\ell \geq 8$.  
 
 Next, it is unclear if the $\mathcal{O}(k^2 \log k)$ bound in Theorem~\ref{thm:main} is optimal.  We conjecture that Theorem~\ref{thm:main} holds with a  $\mathcal{O}(k \log k)$ bounding function.  
 
 \begin{conjecture} \label{klogk}
 There exists a function $f(k)=\mathcal{O}(k \log k)$ such that for every $C_4$-free graph $G$ and every $k \in \N$, $G$ either contains $k$ vertex-disjoint holes of length at least $6$, or a set $X$ of at most $f(k)$ vertices such that $G-X$ has no hole of length at least $6$.
 \end{conjecture}
 
 By the argument in the introduction, if Conjecture~\ref{klogk} is true, then the induced \EP{} property for cycles of length at least $4$ also holds with a $\mathcal{O}(k \log k)$ bounding function.  As noted in~\cite{KK20}, this would be tight via a reduction to the $\Omega(k \log k)$ lower bound that holds for cycles.  
 In general, smaller bounding functions give improved approximation factors for the corresponding packing and covering problems.  For example, the 
 $\mathcal{O}(k \log k)$ bounding function for planar minors proved in~\cite{CHJR19}, currently gives the best approximation algorithm for packing planar minors.  
 
Our final open problem concerns what other classes of graphs have the induced \EP{} property.  
For a fixed graph $H$, the class of \emph{$H$-subdivisions} is the class of graphs that can be obtained from $H$ by repeatedly subdividing edges. 
Theorem~\ref{kimkwon} says that the the class of $C_4$-subdivisions has the induced Erd\H{o}s-P\'osa property, 
and Theorem~\ref{counterexamples} says that for all $\ell \geq 5$, the class of $C_\ell$-subdivisions does not have the induced Erd\H{o}s-P\'osa property.  Kwon and Raymond~\cite{KwonR18} investigated for which graphs $H$ does the class of $H$-subdivisions have the induced Erd\H{o}s-P\'osa property. They designed three different types of constructions to show that for various graphs $H$, the class of $H$-subdivisions does not have the induced Erd\H{o}s-P\'osa property.  In each of their constructions, a large complete bipartite induced subgraph always appears. This is another reason why investigating the induced \EP{} property in $C_4$-free graphs is particularly interesting.  Our main theorem asserts that $C_6$-subdivisions have the induced \EP{} property in the class of $C_4$-free graphs.

\begin{question}
For which graphs $H$ does the class of $H$-subdivisions have the induced \EP{} property?  For which graphs $H$ does the class of $H$-subdivisions have the induced \EP{} property in $C_4$-free graphs?
\end{question}

Note that it is not true that for all graphs $H$, the class of $H$-subdivisions has the induced Erd\H{o}s-P\'osa property in $C_4$-free graphs. Kwon and Raymond~\cite{KwonR18} proved that if a forest $F$ has a connected component having at least two vertices of degree at least $3$, then the class of $F$-subdivisions does not have the induced Erd\H{o}s-P\'osa property in $C_4$-free graphs.

    \bibliographystyle{abbrv}
    \bibliography{longhole}

    \end{document}